\documentclass{amsart}
\usepackage{graphicx} 
\usepackage{amsmath,amssymb,amsfonts}
\usepackage{bm}
\usepackage{esint}
\usepackage{amsthm}
\usepackage{mathrsfs}
\usepackage{xcolor}
\usepackage{comment}
\usepackage{geometry}
\usepackage{verbatim}
\usepackage{graphicx}
\usepackage{color}

\usepackage[colorlinks=true,urlcolor=blue, citecolor=red,linkcolor=blue,
linktocpage,pdfpagelabels, bookmarksnumbered,bookmarksopen]{hyperref}
\usepackage[hyperpageref]{backref}

\usepackage[toc,page]{appendix}
\usepackage{tikz-cd}
\usepackage{tikz}

\newtheorem{theorem}{Theorem}[section]
\newtheorem{lemma}[theorem]{Lemma}
\newtheorem{proposition}[theorem]{Proposition}
\newtheorem{corollary}[theorem]{Corollary}
\theoremstyle{definition}
\newtheorem{remark}[theorem]{Remark}

\newtheorem{definition}[theorem]{Definition}
\newtheorem*{open}{Open problem}
\newtheorem*{ack}{Acknowledgments}

\numberwithin{equation}{section}

\author[Bozzola]{Francesco Bozzola}
\address[F.\ Bozzola]{Dipartimento di Scienze Matematiche, Fisiche e Informatiche
	\newline\indent
	Universit\`a di Parma
	\newline\indent
	Parco Area delle Scienze 53/a, Campus, 43124 Parma, Italy}
\email{francesco.bozzola@unipr.it}

\author[Brasco]{Lorenzo Brasco}
\address[L.\ Brasco]{Dipartimento di Matematica e Informatica
	\newline\indent
	Universit\`a degli Studi di Ferrara
	\newline\indent
	Via Machiavelli 35, 44121 Ferrara, Italy}
\email{lorenzo.brasco@unife.it}

\date{\today}

\title[Topology and capacity]{The role of topology and capacity\\ in some bounds for principal frequencies}


\begin{document}

\begin{abstract}
We prove a lower bound on the sharp Poincar\'e-Sobolev embedding constants for general open sets, in terms of their inradius. We consider the following two situations: planar sets with given topology; open sets in any dimension, under the restriction that points are not removable sets. In the first case, we get an estimate which optimally depends on the topology of the sets, thus generalizing a result by Croke, Osserman and Taylor, originally devised for the first eigenvalue of the Dirichlet-Laplacian.
We also consider some limit situations, like the sharp Moser-Trudinger constant and the Cheeger constant. As a a byproduct of our discussion, we also obtain a Buser--type inequality for open subsets of the plane, with given topology. An interesting problem on the sharp constant for this inequality is presented.
\end{abstract}

\maketitle
\begin{center}
	\begin{minipage}{10cm}
		\small
		\tableofcontents
	\end{minipage}
\end{center}

\section{Introduction}

\subsection{Overview}
For an open set $\Omega\subseteq\mathbb{R}^N$, a natural question in Spectral Geometry is the following: is it possible to bound from below the bottom of its spectrum of the Dirichlet-Laplacian, in terms of the {\it inradius} of the set? The latter is the following simple geometric quantity
\[
r_\Omega=\sup\Big\{r>0\, :\, \exists B_r(x_0)\subseteq\Omega\Big\},
\]
i.e. the supremum of the radii of the balls inscribed in $\Omega$. If we indicate the bottom of the spectrum by 
\[
\lambda(\Omega)=\inf_{u \in C^{\infty}_0(\Omega) \setminus \{0\}} \displaystyle  \frac{\| \nabla u\|^2_{L^2(\Omega)}}{\| u\|^2_{L^2(\Omega)}},
\]
the question above is motivated by the following simple (yet optimal) upper bound
\begin{equation}
\label{minchions}
\lambda(\Omega)\le \frac{\lambda(B_1(0))}{r_\Omega^2}.
\end{equation}
This follows by observing that $\lambda$ is monotone non-increasing with respect to set inclusion, together with its scale properties.
Thus, one would like to know whether (and to which extent) the previous estimate can be reverted or not. 
\par
This is a very classical problem: the answer is well-known to be negative, in such a great generality. The typical example is the ``pepper'' set\footnote{We borrow this fancy terminology from Adams, see for example \cite{Ad}.} $\mathbb{R}^N\setminus \mathbb{Z}^N$, for $N\ge 2$: this has a finite inradius, but $\lambda(\mathbb{R}^N\setminus\mathbb{Z}^N)=0$ since points are removable sets in dimension $N\ge 2$, i.e. they are sets with zero {\it capacity}.
\par
The situation becomes interesting (and the answer to the initial question is positive), provided some geometry comes into play: for example, a lower bound of the type
\begin{equation}
\label{MHOTC}
\lambda(\Omega)\ge \frac{C}{r_\Omega^2},
\end{equation}
holds for convex sets (see \cite[Th\'eor\`eme 8.1]{He} and \cite[Theorem 2.1]{Ka}). More generally, as it is clear from the proof of \cite{Ka}, this is still valid for open sets $\Omega\subseteq\mathbb{R}^N$ such that the {\it distance function}
\[
d_\Omega(x):=\min_{y\in\partial\Omega} |x-y|,\qquad \mbox{ for }x\in\Omega,
\]
is weakly superharmonic in $\Omega$ (see also \cite[Remark 5.8]{BPZ}). These are quite rigid assumptions, but it should be noticed that in general they can not be weakened too much: for example, starting from dimension $N\ge 3$, ``convexity'' can not be replaced by ``starshapedness'', as shown by a simple counterexample in \cite[Section 4]{Ha}. This is due to the fact that lines have zero capacity, when the ambient dimension is at least $3$.
\par
On the other hand, the case $N=2$ is special: in this case, very simple topological assumptions may lead to a positive answer. For example, a remarkable result by Makai \cite{Ma} (neglected for various years and rediscovered independently by Hayman in \cite[Theorem 1]{Ha}) asserts that \eqref{MHOTC} holds for every {\it simply connected} subsets of $\mathbb{R}^2$. Actually, in this very beautiful and striking result, the topological assumption can be further relaxed. The same kind of result still holds for {\it multiply connected} open subsets of $\mathbb{R}^2$. Their precise definition is as follows:
 \begin{definition}\label{def: k connesso con compattificazione}
Let us indicate by $(\mathbb{R}^2)^*$ the {\it one-point compactification} of $\mathbb{R}^2$, i.e. the compact space obtained by adding to $\mathbb{R}^2$ the point at infinity. We say that an open connected set $\Omega\subseteq\mathbb{R}^2$ is {\it multiply connected of order $k$} if its complement in $(\mathbb{R}^2)^*$  has $k$ connected components. When $k=1$, we will simply say that $\Omega$ is {\it simply connected}.
\end{definition}
For this class of planar sets, Taylor in \cite[Theorem 2]{Ta} proved the following lower bound
\[
\lambda(\Omega)\ge \frac{C}{k}\,\left(\frac{1}{r_\Omega}\right)^2.
\]
The constant $C$ can be made explicit, but its sharp value is still unknown. The best known lower bound for the case $k=1$ is due to Ba\~nuelos and Carroll (see \cite[Corollary 1]{BC}). For the general case $k\ge 2$, this is the one obtained by Croke in \cite{Cr}, by refining the method of proof by Osserman \cite{Os}.

However, it is important to notice that the dependence on the ``topological index'' $k$ is optimal, i.e. one can construct sequences of open sets $\{\Omega_k\}_{k\in\mathbb{N}\setminus\{0\}}$ such that $r_{\Omega_k}$ is uniformly bounded, each $\Omega_k$ is multiply connected of order $k$ and 
\[
\lambda(\Omega_k)\sim \frac{1}{k},\qquad \mbox{ as } k\to\infty.
\]
We also refer to \cite[Theorem 3]{GR} for another proof of this result, though the result in \cite{GR} is slightly worse in its dependence on $k$.
\vskip.2cm\noindent
In this paper, we want to extend this kind of analysis to any {\it Poincar\'e-Sobolev embedding constant}, not only to the bottom of the spectrum of the Dirichlet-Laplacian.
More precisely, for $1\le p < \infty$ and $q\ge 1$ such that\footnote{As usual, the number $p^*$ denotes the exponent of the critical Sobolev embedding, defined by 
\[
p^*=\frac{N\,p}{N-p}.
\]}
\[
 \left\{ \begin{array}{ll}
   q\le p^*,& \mbox{ if } 1\le p<N,\\
   q<\infty,& \mbox{ if } p=N,\\
   q\le \infty, & \mbox{ if } p>N,
   \end{array}
   \right.
\]
we want to consider the following quantity
\[
\lambda_{p,q}(\Omega) = \inf_{u \in C^{\infty}_0(\Omega) \setminus \{0\}} \displaystyle  \frac{\| \nabla u\|^p_{L^p(\Omega)}}{\| u\|^p_{L^q(\Omega)}}.
\]
We then seek for lower bounds on this constant, in terms of the inradius only, possibly under some geometric/topological assumptions on the sets.
\par
Observe that if we denote by $\mathscr{D}_0^{1,p}(\Omega)$ the completion of $C^\infty_0(\Omega)$ with respect to the norm \[
\varphi \mapsto \|\nabla \varphi\|_{L^p(\Omega)},
\]
then $\lambda_{p,q}(\Omega)$ is the sharp constant for the embedding $\mathscr{D}^{1,p}_0(\Omega)\hookrightarrow L^q(\Omega)$. 
It may happen that $\lambda_{p,q}(\Omega)=0$: in this case, such an embedding does not hold. 
\par
The quantities $\lambda_{p,q}$ are sometimes called
{\it generalized principal frequencies of the $p-$Laplacian operator with Dirichlet boundary conditions}. In the particular case $q=p$, we will use the shortcut notation 
\[   
\lambda_p(\Omega):=\lambda_{p,p}(\Omega).
\]
For the initial case $p=q=2$, we will still use the distinguished notation $\lambda(\Omega)$.
\par
Occasionally, we will need the space $W^{1,p}_0(\Omega)$: this is the closure of $C^{\infty}_0(\Omega)$ in $W^{1,p}(\Omega)$, endowed with the norm 
\[
\varphi \mapsto \|\varphi\|_{W^{1,p}(\Omega)}:=\left(\|\varphi\|^p_{L^p(\Omega)} + \|\nabla \varphi\|^p_{L^p(\Omega)}\right)^{\frac{1}{p}}.
\]
Note that the spaces $W^{1,p}_0(\Omega)$ and $\mathscr{D}_0^{1,p}(\Omega)$ coincide, whenever $\lambda_{p,q}(\Omega)>0$ for some $1\le q\le p$ (see for example \cite[Proposition 2.4]{BPZ2}). We also recall that the value $\lambda_{p,q}(\Omega)$ is unchanged, if we replace $C^\infty_0(\Omega)$ by its closure $W^{1,p}_0(\Omega)$ (see \cite[Lemma 2.6]{BPZ2}).

\subsection{Description of the results}
We will give two types of results for the problem previously exposed:
\begin{itemize}
\item {\it topological results}, i.e. estimates on $\lambda_{p,q}$ for {\it planar} sets having given topological properties, as in the Croke-Osserman-Taylor result;
\vskip.2cm
\item {\it capacitary results}, i.e. estimates on $\lambda_{p,q}$ for general open sets in any dimension $N$, under the restriction that $p>N$. Under this assumption, points have positive $p-$capacity, that is they are not removable sets for the relevant Sobolev space.
\end{itemize}
Let us briefly describe the results we will present in this paper, by postponing some comments about comparisons with already exhisting results.
\vskip.2cm\noindent
Let $1\le p\le 2$ and let  $p \leq q $ be such that
\[
 \left\{ \begin{array}{ll}
   q<p^*,& \mbox{ if } 1\le p<2,\\
   q<\infty,& \mbox{ if } p=2,\\
   q\le \infty, & \mbox{ if } p>2.
   \end{array}
   \right.
\]
We will show in Theorem \ref{thm:makai-hayman-multiply} that for every $\Omega\subseteq\mathbb{R}^2$ open multiply connected set of order $k \in \mathbb{N} \setminus \{0\}$, we have
\begin{equation}
\label{OTCp}
\lambda_{p,q}(\Omega) \geq  \Theta_{p,q}\,\left(\frac{1}{\sqrt{k}\,r_{\Omega}}\right)^{p-2+\frac{2\,p}{q}}.
\end{equation}
Though not optimal, the constant $\Theta_{p,q}$ is explicit. Moreover, we show that it depends in the correct way on the parameter $q$, as this goes to $p^*$ (case $p<2=N$) or to $\infty$ (case $p=2=N$). We also point out that the dependence on the topology $k$ in the previous estimate is optimal, see Remark \ref{rem:optimal} for these comments.
\vskip.2cm\noindent
For the second type of results, we will show in Theorem \ref{thm:endpoint} that for $p>N$ we have
\begin{equation}
\label{MHp}
\lambda_p(\Omega)\ge C_{N,p}\,\left(\frac{1}{r_\Omega}\right)^p\qquad \mbox{ and }\qquad \lambda_{p,\infty}(\Omega)\ge C_{N,p}\,\left(\frac{1}{r_\Omega}\right)^{p-N}.
\end{equation}
From these two ``endpoints'' estimates, we can then get a related lower bound for $\lambda_{p,q}(\Omega)$ with $p<q<\infty$ (see Corollary \ref{coro:interpol}), just by a simple interpolation argument.
Here as well, even if the constant $C_{N,p}$ obtained is very likely not optimal, we can prove that it exhibits the following asymptotic behaviours
\[
C_{N,p}\sim (p-N)^{p-1},\ \mbox{ as } p\searrow N, \qquad \mbox{ and }\qquad \lim_{p\nearrow \infty} \left(C_{N,p}\right)^\frac{1}{p}=1.
\]
Both behaviours are optimal, see Remark \ref{rem:asympN}. Moreover, our method of proof naturally leads to the appearing of some particular ``punctured'' Poincar\'e constants, which seems to be tightly connected with the determination of the sharp constant in the above lower bounds. Their analysis is interesting in itself. For every $K\subseteq\mathbb{R}^N$ open bounded convex set and every $x_0\in \overline{K}$, these are given by 
\[
\Lambda_p(K\setminus\{x_0\})=\inf_{u\in \mathrm{Lip}(\overline K)} \left\{\int_K |\nabla u|^p\,dx\, :\, \|u\|_{L^p(K)}=1,\, u(x_0)=0\right\},
\]
and
\[
\Lambda_{p, \infty}(K\setminus\{x_0\}) = \inf_{u\in \mathrm{Lip}(\overline K)} \left\{ \int_{K} |\nabla u|^p dx \, : \, \| u \|_{L^\infty(K)} = 1, u(x_0)=0 \right\}.
\]
\begin{remark}[Comparison with previous results I]
The inequality \eqref{OTCp} is a generalization to the case of $\lambda_{p,q}$ of the classical result by Osserman, Taylor and Croke previously mentioned. For the particular case $q=p$, such a generalization has been already obtained by Poliquin in \cite[Theorem 2]{Po2}. Apart for allowing $q\not=p$, our method of proof is different: dissimilarly to Poliquin, who relies on the Osserman-Croke argument, we follow the approach by Taylor.
\par 
While producing a worse constant, Taylor's proof is extremely robust and flexible, just relying on a geometric property of multiply connected sets with finite inradius (what we called ``Taylor's fatness lemma'' in \cite{BB}), together with some properties of $p-$capacity. The method is explained in detail in the introduction of \cite{BB}, where these same ideas are applied to the case of the first eigenvalue of the {\it fractional} Dirichlet-Laplacian. Its simplicity and intrinsically variational nature permit to treat the whole family of $\lambda_{p,q}$ at the same time, without any distinction. 
\par
We point out that with this method, no a priori knowledge of the regularity properties of extremals for $\lambda_{p,q}$ is needed.
On the contrary, in the proof by Osserman and Croke, the main ingredient is given by a suitable Cheeger--type inequality (see \cite[Lemma 2]{Os}). The proof of this inequality relies on a careful analysis on the topology of the level sets of extremals. Extending this technique to the case $p\not=2$ is quite delicate, since in this case extremals are well-known to be only $C^{1,\alpha}$ regular, a property which does not permit to apply\footnote{In dimension $N\ge 2$, we recall that the minimal assumption for the validity of this result is $C^{N-1,1}$ regularity (see \cite[Theorem 1]{Bat} and also \cite{DeP}). For $C^{N-1,\alpha}$ with $\alpha<1$, one can already build counter-examples to Sard's Lemma, see \cite{AL}.} Sard's Lemma. The latter is an essential ingredient in the proof for $p=2$ (where extremals are actually $C^\infty$).
\end{remark}

\begin{remark}[Comparison with previous results II]
For the case $p>N$, the first inequality in \eqref{MHp} has been obtained by Poliquin in \cite[Theorem 1.4.1]{Po1}. Here as well, apart from discussing the whole family of $\lambda_{p,q}$, we use a slightly different argument, which in turn permits to have a better control on the constant. This in turn permits to improve the lower bound given in \cite[Theorem 5.4 \& Remark 5.5]{BPZ2}, where the same estimate is obtained by means of Hardy's inequality: if on the one hand the estimate in \cite{BPZ2} is very simple and explicit, on the other hand it does not display the correct decay rate to $0$, as $p$ goes to $N$. This undesired behaviour is rectified by our proof.
\end{remark}

\subsection{A glimpse of Cheeger's constant} 
Let us now go back to the two-dimensional setting.
In the last part of the paper, we observe that \eqref{OTCp} in the case $p=q=1$ boils down to the following geometric estimate
\[
h(\Omega)\ge \frac{\Theta_{1,1}}{\sqrt{k}\,r_{\Omega}},
\]
where
\[
h(\Omega):=\inf\left\{\frac{\mathcal{H}^{N-1}(\partial E)}{|E|}\, :\, E\Subset\Omega \mbox{ has a smooth boundary}\right\},
\]
is the so-called {\it Cheeger constant}. This result, which is interesting in itself, in turn permits to give a spectral estimate relating the geometric constant $h$ with the bottom of the spectrum $\lambda$. Indeed, by joining this lower bound with \eqref{minchions}, we finally end up with the following upper bound 
\begin{equation}
\label{buserk}
\lambda(\Omega)\le C\,k\,\Big(h(\Omega)\Big)^2.
\end{equation}
This holds again for every $\Omega\subseteq \mathbb{R}^2$ multiply connected open set of order $k$. Such an estimate is better appreciated by recalling the celebrated {\it Cheeger inequality}, i.e. the following spectral lower bound of geometric flavour
\begin{equation}
\label{cheeger}
\left(\frac{h(\Omega)}{2}\right)^2\le \lambda(\Omega),
\end{equation}
which holds for every open set $\Omega\subseteq\mathbb{R}^N$ and every dimension $N$ (see for example \cite[Chapter 4, Section 2]{Maz}). 
\par
Reverting this kind of estimate in general is not possible, unless some severe geometric restrictions are taken: this is possible for convex sets (see \cite[Proposition 4.1]{Pa} and \cite[Corollary 4.1]{BrToulouse}). 
On the contrary, exactly as in the case of the inradius, it fails already for starshaped sets in dimension $N\ge 3$, see \cite[Chapter 4, Section 3]{Maz}. This kind of reverse Cheeger's inequality is also called {\it Buser's inequality}, named after Buser that in \cite{Bu} first obtained this type of estimate, in the framework of Riemannian manifolds (see also Ledoux' papers \cite{Le2, Le}). It is also mandatory to refer to the paper \cite{Mi}.
\par
Our spectral upper bound \eqref{buserk} can thus be regarded as Buser's inequality for multiply connected open sets in the plane. As simple as it is, it is quite remarkable that in dimension $N=2$ this holds without any curvature assumption on the sets. We notice however that the estimate gets spoiled, as the topology of the sets becomes more and more intricate (i.e. as $k$ goes to $\infty$). We show by a suitable family of perforated sets, that this behaviour is ``essentially'' optimal, in the sense that the factor $k$ in \eqref{buserk} can not be replaced by any term of the type $k^\alpha$, for $0<\alpha<1$ (see Proposition \ref{prop:buserconstant}). 

\subsection{Plan of the paper}
As usual, we start by setting the main notation used throughout the paper, together with some preliminary results of general character: this is the content of Section \ref{sec:2}. In Section \ref{sec:3}, we discuss the two-dimensional case, under topological restrictions on the open sets. The case $p>N$ is contained in Section \ref{sec:4}: a good part of the section is devoted to some properties of the ``punctured'' Poincar\'e constants. Finally, in Section \ref{sec:5} we discuss the case of the Cheeger constant and present Buser's inequality for multiply connected sets. A discussion on the optimal constant for this inequality leads to an interesting open problem, which closes this section. Finally, a technical appendix complements the paper.

\begin{ack}
We thank Gian Paolo Leonardi for his kind interest towards our work. We also thank Michele Miranda for putting us in contact. Part of this paper has been written during the workshop ``{\it Shape Optimization and Isoperimetric and Functional Inequalities\,}'', held in Levico Terme in September 2023. 
Organizers and hosting institutions are gratefully acknowledged.
\par
F.\,B. has been financially supported by the joint Ph.D. program of the Universities of Ferrara, Modena \& Reggio Emilia and Parma. L.\,B. has been financially supported by the {\it Fondo di Ateneo per la Ricerca} {\sc FAR 2020} and {\sc FAR 2021} of the University of Ferrara. 
\end{ack}

\section{Preliminaries}
\label{sec:2}

\subsection{Notation}
In the sequel, for $d>0$ and $x_0\in\mathbb{R}^N$ we will use the notation
\[
Q_d(x_0)=\prod_{i=1}^N (x_0^i-d,x_0^i+d),\qquad \mbox{ where } x_0=(x_0^1,\dots,x_0^N),
\]
for an open $N-$dimensional cube, centered at $x_0$ and having sides parallel to the coordinate axes. When $x_0$ coincides with the origin, we will simply write $Q_d$.
\par
We will denote by $B_R(x_0)$ the $N-$dimensional open ball centered at $x_0$, with radius $R>0$. As above, we will simply write $B_R$ when the ball is centered at the origin. By $\omega_N$ we will indicate the volume of the $N-$dimensional ball $B_1$.
\par
For every $k\in\mathbb{N}$, by the symbol $\mathcal{H}^k$ we will denote the $k-$dimensional Hausdorff measure.
\par
Let $1\le p <\infty$, for every open bounded set $E\subseteq\mathbb{R}^N$ and every compact set $K \Subset E$, we define the {\it $p-$capacity of $K$ relative to $E$} as
\begin{equation}
\label{capacity}
\mathrm{cap}_p(K; E) = \inf_{u\in C^\infty_0(E)} \left\{\int_E |\nabla u|^p\,dx \, : \,  u \geq 1_K \right\}.
\end{equation}
By a standard approximation argument, such an infimum is unchanged, if we enlarge the class of admissible functions to Lipschitz ones, compactly supported in $E$ and such that $u\ge 1_K$. We refer the reader to \cite[Chapter 2, Section 2]{Maz} for a thorough study of the properties of $p-$capacity.

\subsection{An extension operator}

The next Lemma states the existence of an extension operator for Sobolev functions defined on a ball, with a precise control on the extension constants. This is taken from \cite[Proposition 3.1]{BB}. The extension operator is obtained by simply composing functions with the \textit{inversion with respect to $\mathbb{S}^{N-1}$}, i.e. the $C^1$ bijection $\mathcal{K}:\mathbb{R}^N\setminus\{0\}\to \mathbb{R}^N\setminus\{0\}$, given by
	\[
	\mathcal{K}(x)=\frac{x}{|x|^2},\qquad \mbox{ for every }x\in\mathbb{R}^N\setminus\{0\}.
	\]
\begin{lemma} \label{op_est_palle}
	Let $x_0 \in \mathbb{R}^N$ and $r > 0$. There exists a linear extension operator 
	\[
	\mathcal{E}_r : L^1(B_r(x_0)) \rightarrow L^1_{\rm loc}(\mathbb{R}^N), 
	\]
	such that, for every $1\le p\le \infty$, it maps $W^{1,p}(B_r(x_0))$ to $W^{1,p}_{\rm loc}(\mathbb{R}^N)$. 	Moreover, for every $u\in W^{1,p}(B_r(x_0))$ and every $R > r$, it holds  \begin{equation} 
	\label{stima_1_estensione}
\big\| \mathcal{E}_r[u]\big\|_{L^p(B_R(x_0))} \leq 2^{\frac{1}{p}}\, \left(\frac{R}{r} \right)^{\frac{2N}{p}}\, \|u\|_{L^p(B_r(x_0))}, 
\end{equation}		
\begin{equation}		
\label{stima_2_estensione}
		\big\|\nabla \mathcal{E}_r[u]\big\|_{L^p(B_R(x_0))} \leq 4^{\frac{1}{p}} \,\left(\frac{R}{r} \right)^{\frac{4N}{p}}\, \|\nabla u\|_{L^p(B_r(x_0))}.
	\end{equation}
\end{lemma}

\begin{proof}
	The case $1 < p \leq \infty$ is already covered by \cite[Proposition 3.1]{BB}. The case $p=1$ was not considered there, since it was not needed, but we can easily fill the gap. Without loss of generality, we can suppose that $x_0$ coincides with the origin and that $r=1$. For every $u\in L^1(B_1)$, we recall that the extension $\mathcal{E}_r[u]$ is given by
	\begin{equation}\label{def trasformata di kelvin u}
	\mathcal{E}_r[u](x)=\left\{\begin{array}{ll}
	u(x),&\mbox{ if }x\in B_1,\\
	u(\mathcal{K}(x)),&\mbox{ if }x\in \mathbb{R}^N\setminus B_1,
	\end{array}
	\right.
	\end{equation}
It is easily seen that if $x\in B_R\setminus B_1$, then $\mathcal{K}(x)\in B_1\setminus B_{1/R}$. Moreover, we have  
\[
\mathcal{K}^{-1}(x)=\mathcal{K}(x)\quad \mbox{ and }\quad |\mathrm{det}(D\mathcal{K}(x))|=\frac{1}{|x|^{2\,N}},\qquad \mbox{ for every }x\in\mathbb{R}^N\setminus \{0\}.
\]
The estimate \eqref{stima_1_estensione} for the $L^1$ norm is readily obtained: for every $R>1$, thanks to the properties of $\mathcal{K}$ we have
\[
\begin{split}
\big\|\mathcal{E}_r[u]\big\|_{L^1(B_R)}&=\int_{B_R\setminus B_1} |u(\mathcal{K}(x))|\,dx+\int_{B_1} |u|\,dx\\
&=\int_{B_1\setminus B_{1/R}} |u(y)|\,|\mathrm{det}(D\mathcal{K}^{-1}(y))|\,dy+\int_{B_1} |u|\,dx\\
&\le (R^{2\,N}+1)\,\int_{B_1} |u|\,dx\le 2\,R^{2\,N}\,\int_{B_1} |u|\,dx.
\end{split}
\]
We now have to show that if $u\in W^{1,1}(B_1)$, then $\mathcal{E}[u]\in W^{1,1}_{\rm loc}(\mathbb{R}^N)$ and the estimate \eqref{stima_2_estensione} holds.
By classical approximation results (see for example \cite[Theorem 3.6]{Gi}), there exists a sequence $\{u_n\}_{n\in\mathbb{N}}\subseteq C^1(\overline{B_1})$ such that
\[
\lim_{n\to\infty} \|u_n-u\|_{L^1(B_1)}=\lim_{n\to\infty} \|\nabla u_n-\nabla u\|_{L^1(B_1)}=0.
\]
By using the extension result for $p>1$ and the linearity of $\mathcal{E}_r$, for every $1<p<\infty$ and every $n,m\in\mathbb{N}$ we have 
\[
\big\|\nabla \mathcal{E}_r[u_n]-\nabla \mathcal{E}_r[u_m]\big\|_{L^p(B_R)}\le 4^{\frac{1}{p}}\,R^{\frac{4N}{p}}\, \|\nabla u_n-\nabla u_m\|_{L^p(B_1)}.
\]
For every $n,m\in\mathbb{N}$ we can take the limit as $p$ goes to $1$ in the previous estimate. In conjunction with \eqref{stima_1_estensione} for $p=1$, we get	
	\[
\big\|\mathcal{E}_r[u_n]-\mathcal{E}_r[u_m]\big\|_{L^1(B_R)}\leq 2 \,R^{2N}\, \|u_n-u_m\|_{L^1(B_r)},
	\]
and
\[
\big\|\nabla \mathcal{E}_r[u_n]-\nabla \mathcal{E}_r[u_m]\big\|_{L^1(B_R)}\le 4\, R^{4N}\, \|\nabla u_n-\nabla u_m\|_{L^1(B_r)}.
\]
These entail that $\{\mathcal{E}_r(u_n)\}_{n\in\mathbb{N}}\subseteq W^{1,1}(B_R)$ is a Cauchy sequence. Observe that this property holds for every finite $R>r$.
Thus, for every $R>r$ this sequence converges to a limit function, that we indicate by $U_R\in W^{1,1}(B_R)$. On the other hand, for every $R>1$ and $n\in\mathbb{N}$ we have
\[
\begin{split}
\big\|\mathcal{E}_r[u]-U_R\big\|_{L^1(B_R)}&\le \big\|\mathcal{E}_r[u]-\mathcal{E}_r[u_n]\big\|_{L^1(B_R)}+\big\|\mathcal{E}_r[u_n]-U_R\big\|_{L^1(B_R)} \\
&\le 2 \,R^{2N}\,\|u-u_n\|_{L^1(B_1)}+\big\|\mathcal{E}_r[u_n]-U_R\big\|_{L^1(B_R)},
\end{split}
\]
where we used \eqref{stima_1_estensione} for $u-u_n$, with $p=1$. By taking the limit as $n$ goes to $\infty$, we get $\mathcal{E}_r[u]=U_R\in W^{1,1}(B_R)$ almost everywhere in $B_R$. This shows that $\mathcal{E}_r[u]\in W^{1,1}_{\rm loc}(\mathbb{R}^N)$. The estimate \eqref{stima_2_estensione} is then obtained by approximation.
\end{proof}

By joining the previous result and the fact that each open bounded convex set $K\subseteq\mathbb{R}^N$ is bi-Lipschitz homeomorphic to a ball, we can obtain an extension operator for functions defined on $K$. This is taken from \cite[Section 3]{BB}, as well.
\begin{corollary}
\label{coro:estensione}
Let $K\subseteq\mathbb{R}^N$ be an open bounded convex set and $x_0\in K$, there exists a linear extension operator 
	\[
	\mathcal{E}_K:L^1(K)\to L^1_{\rm loc}(\mathbb{R}^N),
	\]
	such that,
	for every $1\le  p\le\infty$, it maps $W^{1,p}(K)$ to $W^{1,p}_{\rm loc}(\mathbb{R}^N)$. Moreover, if we define the following scaled copy of $K$
\[
	K_R(x_0):=R\,(K-x_0)+x_0=\Big\{R\,(x-x_0)+x_0\,:\,x\in K\Big\},
	\]
for every $u\in W^{1,p}(K)$ and every $R>1$ we have
		\begin{equation}
		\label{estensioneK1}
	\big\|\nabla \mathcal{E}_K[u]\big\|_{L^p(K_R(x_0))}\le \mathcal{A}\,R^\frac{4\,N}{p}\,\|\nabla u\|_{L^p(K)},
	\end{equation} 
	and
	\begin{equation}
	\label{estensioneK2}
	\big\|\mathcal{E}_K[u]\big\|_{L^p(K_R(x_0))}\le \mathcal{B}\,R^\frac{2\,N}{p}\,\|u\|_{L^p(K)}.
	\end{equation}
	The constants $\mathcal{A}=\mathcal{A}(N,p,K,x_0)>0$ and $\mathcal{B}=\mathcal{B}(N,p,K,x_0)>0$ are given by
	\[
	\mathcal{A}=\big(4\cdot 6^{3\,N+p}\big)^\frac{1}{p}\,\left(\frac{D_K(x_0)}{d_K(x_0)}\right)^{\frac{6\,N}{p}+2}\qquad \mbox{ and }\qquad	\mathcal{B}=\big(2\cdot 6^N)^\frac{1}{p}\,\left(\frac{D_K(x_0)}{d_K(x_0)}\right)^\frac{2\,N}{p},
	\]
	where 
		\[
d_K(x_0)=\min_{x\in\partial K} |x-x_0|,\qquad D_K(x_0)=\max_{x\in \partial K} |x-x_0|.
\]
\end{corollary}

\subsection{Poincar\'e-Wirtinger inequalities}

In this paper, we will occasionally need also the sharp constants for some Poincar\'e-Wirtinger--type inequalities. More precisely, for an open bounded Lipschitz set $\Omega\subseteq\mathbb{R}^N$,
for $1 \le p < \infty$ and 
  \[
 \left\{ \begin{array}{ll}
   1\le q<p^*,& \mbox{ if } 1\le p<N,\\
   1\le q<\infty,& \mbox{ if } p=N,\\
   1\le q\le \infty, & \mbox{ if } p>N,
   \end{array}
   \right.
   \] 
we introduce the quantity
\begin{equation}
\label{muuu}
\mu_{p,q}(\Omega) = \inf \left\{ \displaystyle \frac{\| \nabla u \|^p_{L^p(\Omega)}}{ \displaystyle \min_{t \in \mathbb{R}} \| u - t\|^p_{L^q(\Omega)}} \, : \, u \in \mathrm{Lip}(\overline{\Omega}) \mbox{ is not constant}\right\}.
\end{equation}
In the case $q=p$, we will simply use the symbol $\mu_p(\Omega)$.
\par
We notice that $\mu_{p,q}(\Omega)>0$ if and only if $\Omega$ supports a Poincar\'e-Wirtinger inequality of the form 
	\[
	C\, \min_{t \in \mathbb{R}} \|u - t\|_{L^{q}(\Omega)}^p \leq \| \nabla u\|^p_{L^p(\Omega)},\qquad \mbox{ for every } u\in \mathrm{Lip}(\overline\Omega),
	\]
	for some $C > 0$. In this case, we have $\mu_{p,q}(\Omega) \ge C$ and $\mu_{p,q}(\Omega)$ is the sharp constant in such an inequality.
\begin{remark}
\label{rem:uniquemean}
We recall that for every $u\in \mathrm{Lip}(\overline\Omega)$ and every $1\le q\le\infty$, there exists a unique minimizer $t_u$ of the function
\[
t\mapsto \|u - t\|_{L^{q}(\Omega)}.
\]
For $1<q<\infty$, this is characterized by the following optimality condition
\[
\int_\Omega |u-t_u|^{q-2}\,(u-t_u)\,dx=0.
\]
In the limit case $q=\infty$, this is given by
\[
t_u=\frac{1}{2}\,\sup_{\overline\Omega} u+\frac{1}{2}\,\inf_{\overline\Omega} u.
\]
Finally, in the limit case $q=1$, the optimal $t_u$ coincides with the unique value\footnote{In this case, $t_u$ is called the {\it median} of $u$. Its uniqueness is due to the continuity of $u$: for a discontinuous function, it is easily seen that medians may not be unique.} $t$ such that
\[
\Big|\Big\{x\in\overline\Omega\, :\, u(x)\ge t\Big\}\Big|=\Big|\Big\{x\in\overline\Omega\, :\, u(x)\le t\Big\}\Big|.
\]
We refer to \cite[Theorem 2.1]{IMW} for these facts.
\par
Accordingly, in the case $1<q<\infty$ the constant $\mu_{p,q}(\Omega)$ can be equivalently rewritten as
\[
\mu_{p,q}(\Omega) = \inf \left\{ \displaystyle \frac{\|\nabla u\|^p_{L^p(\Omega)} }{\|u\|^p_{L^q(\Omega)}} \, : \, \int_{\Omega} |u|^{q-2}\, u\, dx = 0, \, u \in \mathrm{Lip}(\overline{\Omega})\backslash \{0\} \right\}.
\]
\end{remark}	
In the sequel, we will need the following geometric lower bound on $\mu_{p,q}$ for convex sets, which is quite classical. In general, this estimate is not sharp, but it will be largely sufficient for our scopes. We refer to \cite{BNT, CW, ENT} and \cite{FNT} for some finer estimates.
\begin{lemma}
\label{lemma:lower_neumann}
Let $1\le p<\infty$ and $q\ge p$ be such that
      \[
 \left\{ \begin{array}{ll}
   q<p^*,& \mbox{ if } 1\le p<N,\\
   q<\infty,& \mbox{ if } p=N,\\
   q\le \infty, & \mbox{ if } p>N.
   \end{array}
   \right.
   \] 
   For every $\Omega\subseteq\mathbb{R}^N$ open bounded convex set, we have
\begin{equation}
	\label{lowerbound_mu}
	\mu_{p,q}(\Omega)\ge \Big(N\,\omega_N^\frac{1}{N}\Big)^p\,\left(\frac{|\Omega|}{\mathrm{diam}(\Omega)^N}\right)^p\, \left(\frac{\dfrac{1}{N}-\dfrac{1}{p}+\dfrac{1}{q}}{1-\dfrac{1}{p}+\dfrac{1}{q}}\right)^{p-1+\frac{p}{q}}\, |\Omega|^{1-\frac{p}{N}-\frac{p}{q}}.
	\end{equation}	
\end{lemma}
\begin{proof}
We take $u\in\mathrm{Lip}(\overline\Omega)$, 
it is sufficient to combine \cite[Lemma 7.12]{GT} 
	and \cite[Lemma 7.16]{GT}. 
This leads to
	\[
	\begin{split} 
	\left\|u - \frac{1}{|\Omega|}\,\int_{\Omega} u\,dy\right\|_{L^q(\Omega)}
		&\leq \frac{1}{N\,\omega_N^{1/N}}\,\frac{\mathrm{diam}(\Omega)^N}{|\Omega|} \left(\frac{1-\dfrac{1}{p}+\dfrac{1}{q}}{\dfrac{1}{N}-\dfrac{1}{p}+\dfrac{1}{q}}\right)^{1-\frac{1}{p}+\frac{1}{q}}\, |\Omega|^{\frac{1}{N}-\frac{1}{p}+\frac{1}{q}}\, \| \nabla u\|_{L^p(\Omega)}.
	\end{split}
	\]
By simply noticing that
\[
\min_{t \in \mathbb{R}} \|u - t\|_{L^{q}(\Omega)}\le \left\|u - \frac{1}{|\Omega|}\,\int_{\Omega} u\,dy\right\|_{L^q(\Omega)},	
\]
we obtain the claimed lower bound.
\end{proof}

\subsection{A Maz'ya-Poincar\'{e} type inequality}
The first cornerstone of our main results is the following Maz'ya--type inequality, for functions defined on a closed cube and vanishing in a (relative) neighborhood of a compact subset. For the proof of such result, we strictly follow \cite[Chapter 14, Section 1.2]{Maz}, up to some minor modifications. We will also give an explicit value for the constant appearing in the estimate (see Remark \ref{remark:costante} below).
\begin{theorem} 
\label{thm:mazya_poincare}
  Let $1\le p \leq q$, with 
       \[
 \left\{ \begin{array}{ll}
   q<p^*,& \mbox{ if } 1\le p<N,\\
   q<\infty,& \mbox{ if } p=N,\\
   q\le \infty, & \mbox{ if } p>N,
   \end{array}
   \right.
   \] 
and let $\Sigma \subseteq \overline{Q_d(x_0)}$ be a compact set.
 Then, for every $D > \sqrt{N}d$ there exists a constant $\mathscr{C} = \mathscr{C} \left(N,p, q, D/d\right) > 0$ such that 
     \[
         \frac{\mathscr{C}}{ d^{\frac{N}{q}}}\ \Big(\mathrm{cap}_p(\Sigma; B_D(x_0))\Big)^{\frac{1}{p}}\, \|u\|_{L^q(Q_d(x_0))} \leq   \|\nabla u\|_{L^p(Q_{d}(x_0))},
    \]
    for every $u \in C^\infty(\overline{Q_d(x_0)})$ with $\mathrm{dist}(\mathrm{supp\,}u, \Sigma) > 0$.
\end{theorem}
\begin{proof}
We can assume that $x_0 = 0$. Let $u \in C^{\infty}(\overline{Q_d})$ be as in the statement, without loss of generality we can also suppose that 
    \begin{equation} 
    \label{cond_1_mazya_type}
\|u\|_{L^q(Q_d)} = |Q_d|^\frac{1}{q}=(2\,d)^\frac{N}{q}.
    \end{equation}
 We use the standard convention that the right-hand side is $1$, in the limit case $q=\infty$.
    Hence, we consider the function  
    \[
    \widetilde{u} := \mathcal{E}_{Q_d}[u],
    \] 
   i.e. the extended function provided by Corollary \ref{coro:estensione}, with $K=Q_d$ and $x_0=0$. For every $D > d$, by applying formula \eqref{estensioneK2} with $R=D/d$, we get
   \begin{equation} 
   \label{prop_estensioneK1}
    	\|\nabla \widetilde{u}\|_{L^p(B_D)} \leq \|\nabla \widetilde{u}\|_{L^p(Q_D)} \leq \mathcal{A}\, \left(\frac{D}{d}\right)^{\frac{4\,N}{p}}\, \|\nabla u\|_{L^p(Q_d)}.
    \end{equation} 
We observe that with this choice for $K$ and $x_0$, we have $D_K(x_0)/d_K(x_0) = \sqrt{N}$, thus the constant $\mathcal{A}$ only depends on $N$ and $p$. More precisely, it is given by
\begin{equation}
\label{alfa}
\mathcal{A}=\big(4\cdot 6^{3\,N+p}\big)^\frac{1}{p}\,\left(\sqrt{N}\right)^{\frac{6\,N}{p}+2}=:\alpha_{N,p}.
\end{equation}
We now fix $D > \sqrt{N}\,d$ as in the statement and let  $\eta$ be a Lipschitz continuous cut-off function compactly supported in $B_D$, such that 
\begin{equation} 
\label{cut_off_prop}
        0 \leq \eta \leq 1, \qquad \eta \equiv 1 \mbox{ on } B_{\sqrt{N}\,d},\qquad \eta\equiv 0 \mbox{ on } B_{D}\setminus B_{\frac{\sqrt{N}\,d+D}{2}}, \quad |\nabla \eta| \leq \frac{2}{D -\sqrt{N}d}.
    \end{equation} 
Then, the function
 \[
    \psi := \eta\,(1-\widetilde{u}) ,
    \]
    is Lipschitz continuous, compactly supported in $B_D$ and such that $\psi\ge 1_\Sigma$, by construction. Thus, it is an admissible function to test the definition of relative $p-$capacity \eqref{capacity}. By the triangle inequality and the properties \eqref{cut_off_prop} of $\eta$, this yields
  \[
  \begin{split}
    	\Big(\mathrm{cap}_p(\Sigma; B_D)\Big)^{\frac{1}{p}} \leq \|\nabla \psi\|_{L^p(B_D)} \leq \|\nabla \widetilde{u}\|_{L^p(B_D)}  + \frac{2}{D-\sqrt{N}d}\,\|1-\widetilde{u}\|_{L^p(B_D)}. 
     \end{split}
     \]
We now denote by $\widetilde{t}$ the unique real number (recall Remark \ref{rem:uniquemean} above) such that 
	\[
    	\| \widetilde{u} - \widetilde{t} \|_{L^q(B_D)} = \min_{t \in \mathbb{R}} \| \widetilde{u} - t\|_{L^q(B_D)}.
    	\]
	Without loss of generalization, we can suppose that 
    \begin{equation} \label{cond_2_mazya_type}
    	\widetilde{t} \geq 0.
    \end{equation}
By a further application of the triangle inequality, we obtain
\[
\Big(\mathrm{cap}_p(\Sigma; B_D)\Big)^{\frac{1}{p}}\leq \|\nabla \widetilde{u}\|_{L^p(B_D)} + \frac{2}{D-\sqrt{N}d}\,\|1-\widetilde{t}\|_{L^p(B_D)} + \frac{2}{D-\sqrt{N}d}\,\|\widetilde{t}-\widetilde{u}\|_{L^p(B_D)}. 
\]    
We have to estimate the last two $L^p$ norms. Actually, the first one can be estimated in terms of the second one. Indeed, by using  \eqref{cond_2_mazya_type} and \eqref{cond_1_mazya_type}, we get
\[
\begin{split}
\|1-\widetilde{t}\|_{L^p(B_D)}=|1-\widetilde{t}|\,|B_D|^\frac{1}{p}&=\left|\|u\|_{L^q(Q_d)}-\|\widetilde{t}\|_{L^q(Q_d)}\right|\,\frac{|B_D|^\frac{1}{p}}{|Q_d|^\frac{1}{q}}\\
&\le \|u-\widetilde{t}\|_{L^q(Q_d)}\,\frac{|B_D|^\frac{1}{p}}{|Q_d|^\frac{1}{q}}\le \|\widetilde{u}-\widetilde{t}\|_{L^q(B_D)}\,\frac{|B_D|^\frac{1}{p}}{|Q_d|^\frac{1}{q}}.
\end{split}
\]
By inserting this estimate in the inequality above, we get
\[
\begin{split}
\Big(\mathrm{cap}_p(\Sigma; B_D)\Big)^{\frac{1}{p}}&\le \|\nabla \widetilde{u}\|_{L^p(B_D)}+\frac{2\,|B_D|^\frac{1}{p}}{D-\sqrt{N}d}\,\left(\frac{1}{|Q_d|^\frac{1}{q}}+\frac{1}{|B_D|^\frac{1}{q}}\right)\,\|\widetilde{t}-\widetilde{u}\|_{L^q(B_D)}.
\end{split}
\]
Moreover, by recalling the definition of $\mu_{p,q}(B_D)$ and the definition of $\widetilde{t}$, we have 
\[
\|\widetilde{t}-\widetilde{u}\|_{L^q(B_D)}\le \left(\frac{1}{\mu_{p,q}(B_D)}\right)^\frac{1}{p}\, \|\nabla \widetilde u\|_{L^p(B_D)}.
\]     
    We thus obtain
    \[
\Big(\mathrm{cap}_p(\Sigma; B_D)\Big)^{\frac{1}{p}}\le \left[1+\frac{2\,\omega_N^\frac{1}{p}}{D-\sqrt{N}d}\,\left(\frac{1}{|B_D|^\frac{1}{q}}+\frac{1}{|Q_d|^\frac{1}{q}}\right)\,\left(\frac{D^{p+\frac{N}{q}p}}{\mu_{p,q}(B_1)}\right)^\frac{1}{p}  \right]\,\|\nabla \widetilde{u}\|_{L^p(B_D)}.  
    \] 
We make some small manipulations, in order to simplify the expression of the constant: we have
\[
\begin{split}
1+\frac{2\,\omega_N^\frac{1}{p}}{D-\sqrt{N}d}&\,\left(\frac{1}{|B_D|^\frac{1}{q}}+\frac{1}{|Q_d|^\frac{1}{q}}\right)\,\left(\frac{D^{p+\frac{N}{q}p}}{\mu_{p,q}(B_1)}\right)^\frac{1}{p} \\
&=\frac{1}{|Q_d|^\frac{1}{q}}\,\left[|Q_d|^\frac{1}{q}+\frac{2\,\omega_N^\frac{1}{p}}{D-\sqrt{N}d}\,\left(\frac{|Q_d|^\frac{1}{q}}{|B_D|^\frac{1}{q}}+1\right)\,\left(\frac{D^{p+\frac{N}{q}p}}{\mu_{p,q}(B_1)}\right)^\frac{1}{p}  \right]\\
&\le\frac{1}{(2\,d)^\frac{N}{q}}\,\left[\omega_N^\frac{1}{q}\,D^\frac{N}{q}+\frac{4\,\omega_N^\frac{1}{p}}{D-\sqrt{N}d}\,\left(\frac{D^{p+\frac{N}{q}p}}{\mu_{p,q}(B_1)}\right)^\frac{1}{p}  \right]\\
&= \left(\frac{D}{2\,d}\right)^\frac{N}{q}\,\left[\omega_N^\frac{1}{q}+\frac{4\,\omega_N^\frac{1}{p}}{1-(\sqrt{N}d)/D}\,\left(\frac{1}{\mu_{p,q}(B_1)}\right)^\frac{1}{p}  \right].
\end{split}
\]
Thus, by recalling the normalization condition \eqref{cond_1_mazya_type}, we have obtained
\[
\frac{1}{d^\frac{N}{q}}\,\Big(\mathrm{cap}_p(\Sigma; B_D)\Big)^{\frac{1}{p}}\,\|u\|_{L^q(Q_d)}\le \left(\frac{D}{d}\right)^\frac{N}{q}\,\left[\omega_N^\frac{1}{q}+\frac{4\,\omega_N^\frac{1}{p}}{1-(\sqrt{N}d)/D}\,\left(\frac{1}{\mu_{p,q}(B_1)}\right)^\frac{1}{p}  \right]\,\|\nabla \widetilde{u}\|_{L^p(B_D)}.
\]
At last, by using \eqref{prop_estensioneK1} in the right-hand side, we get the desired conclusion. 
\end{proof}
\begin{remark}
\label{remark:costante}
By inspecting the proof, we see that the constant $\mathscr{C}$ obtained in the previous theorem has the following explicit expression
\[
\mathscr{C}=\frac{1}{\alpha_{N,p}}\, \left(\frac{d}{D}\right)^{\frac{4\,N}{p}+\frac{N}{q}}\,\left[\omega_N^\frac{1}{q}+\frac{4\,\omega_N^\frac{1}{p}}{1-\dfrac{\sqrt{N}d}{D}}\,\left(\frac{1}{\mu_{p,q}(B_1)}\right)^\frac{1}{p}  \right]^{-1}.
\]
The constant $\alpha_{N,p}$ is given in \eqref{alfa} and it comes from the extension operator.
Actually, the constant $\mu_{p,q}(B_1)$ may not look so explicit: however, it can be conveniently estimated from below by Lemma \ref{lemma:lower_neumann}, in terms of quantities only depending on $N$, $p$ and $q$.
\end{remark}

\section{The case $N=2$}
\label{sec:3}

\subsection{Three technical facts}
We recall the following geometric result due to Taylor (see \cite[proof of Theorem 2]{Ta}). In the form below, this can be found in \cite[Lemma 2.1]{BB}. For every $\alpha\in\mathbb{R}$, we denote by 
\[
\big\lfloor\alpha\big\rfloor=\max\Big\{n\in\mathbb{Z}\, :\, \alpha\ge n\Big\},
\]
its {\it integer part}. For every direction $\omega\in\mathbb{S}^{N-1}$, we also use the notation $\Pi_{\omega}$ for the orthogonal projection onto the space $\langle \omega\rangle^\bot:=\{x\in\mathbb{R}^N\, :\, \langle x,\omega\rangle=0\}$. 
\begin{lemma}[Taylor's fatness Lemma] \label{taylor_lemma}
	Let $k\in\mathbb{N}\setminus\{0\}$ and let $\Omega\subseteq\mathbb{R}^2$ be an open multiply connected set of order $k$, with finite inradius. Let $Q$ be an open square with side length $10\,(\lfloor\sqrt{k}\rfloor+1)\,r_\Omega$, whose
sides are parallel to the coordinate axes.	
Then, there exists a compact set $\Sigma\subseteq\overline{Q}\setminus\Omega$ such that 
\[
\max\Big\{\mathcal{H}^1(\Pi_{\mathbf{e}_1}(\Sigma)),\, \mathcal{H}^1(\Pi_{\mathbf{e}_2}(\Sigma))\Big\}\ge \frac{\sqrt{k}}{4}\,r_\Omega,
\]
where $\mathbf{e}_1=(1,0)$ and $\mathbf{e}_2=(0,1)$.
\end{lemma}
We need also the following simple result.
\begin{lemma} \label{Lemma_cap_point}
   Let $(a, b) \subseteq \mathbb{R}$ and $a<x_0<b$, then for every $p\ge 1$ we have
   \begin{equation*}
       \mathrm{cap}_p(\{x_0\}; (a,b)) \geq \frac{2^p}{(b-a)^{p-1}}.
   \end{equation*}
\end{lemma} 
\begin{proof}
    Let $\psi \in C^\infty_0((a,b))$ such that $\psi(x_0) \geq 1$, then 
 \[
        \int_a^b |\psi'| \,dx = \int_a^{x_0} |\psi'|\,dx + \int_{x_0}^b |\psi'|\,dx \ge |\psi(x_0)- \psi(a)| + |\psi(b) - \psi(x_0)|= 2.
 \]
 By Jensen's inequality, we obtain 
 \[
         \frac{1}{b-a} \int_a^b |\psi'|^p dx \geq \left(\frac{1}{b-a} \int_a^b |\psi'|\, dx \right)^p \geq \frac{2^p}{(b-a)^p}.
 \]
By recalling the definition \eqref{capacity}, the claimed inequality easily follows.
\end{proof}
As a last ingredient, we need a geometric lower bound for $\mathrm{cap}_p(\Sigma; B_r(x_0))$, in the plane. This is the content of the following result, which can be proved along the lines of \cite[Chapter 13, Section 1.2, Proposition 1]{Maz}.
\begin{lemma}[Capacity and projections] 
\label{cap_length}
   Let $\Sigma \Subset B_r(x_0) \subseteq \mathbb{R}^2$ be a compact set. 
   Then, for every $1\le p<\infty$ and every $\omega\in\mathbb{S}^1$ it holds 
\[
       \mathrm{cap}_p(\Sigma; B_r(x_0)) \geq \frac{2}{r^{p-1}}\, \mathcal{H}^1(\Pi_{\omega}(\Sigma)),
\]
where, as above, $\Pi_\omega$ is the orthogonal projection onto $\langle \omega\rangle^\bot=\{x\in\mathbb{R}^2\, :\, \langle x,\omega\rangle=0\}$.        
\end{lemma}
\begin{proof}
It is not restrictive to suppose that $x_0 =0$. We fix $\omega\in\mathbb{S}^1$ and choose $\omega^\bot\in \mathbb{S}^1$ to be orthogonal to it. 
We can also assume that $\mathcal{H}^1(\Pi_{\omega}(\Sigma)) > 0$, otherwise there is nothing to prove.
\par Fix $p \geq 1$ and take any function $u \in C^\infty_0(B_r)$ such that $u \geq 1_\Sigma$. Let $Q$ be the square centered at the origin, with side length $2\,r$ and whose sides are parallel to $\omega$ and $\omega^\bot$. By Fubini's Theorem and writing every $x\in Q$ as follows
\[
x=z_1\,\omega+z_2\,\omega^\bot,  \qquad \mbox{ for } (z_1,z_2)\in(-r,r)\times(-r,r),  
\]
we have
    \begin{align*}
      \int_{B_r} |\nabla u|^p\,dx= \int_{Q} |\nabla u|^p\, dx &\geq \int_{Q} |\partial_\omega u|^p\, dx = \int_{-r}^{r} \int_{-r}^{r} |\partial_{z_1} u(z_1,z_2)|^p\, dz_1 dz_2 \\ &\geq \int_{\Pi_{\omega}(\Sigma)} \|\partial_{z_1} u(\cdot, z_2)\|^p_{L^p((-r,r))}\, dz_2. 
    \end{align*}
By using that for every $z_2\in \Pi_\omega(\Sigma)$, the function $z_1\mapsto u(z_1,z_2)$ is admissible for the definition of the $p-$capacity of a point relative to the interval $(-r,r)$, from Lemma \ref{Lemma_cap_point}, we get
\[
\int_{\Pi_{\omega}(\Sigma)} \|\partial_{z_1} u(\cdot, z_2)\|^p_{L^p((-r,r))}\, dz_2\geq \frac{2^p}{(2\,r)^{p-1}} \mathcal{H}^1(\Pi_{\omega}(\Sigma))= \frac{2}{r^{p-1}} \mathcal{H}^1(\Pi_{\omega}(\Sigma)).
\]
    This concludes the proof.
\end{proof}
\subsection{Inradius and principal frequencies}
We are ready to adapt Taylor's proof and prove the announced lower bound for multiply connected open sets in the plane. We can cover the case of any generalized principal frequency, at the same price.
\begin{theorem} 
\label{thm:makai-hayman-multiply}
 Let $1\le p<\infty$ and let  $p \leq q $ be such that
 \begin{equation}
 \label{esponenti}
 \left\{ \begin{array}{ll}
   q<p^*,& \mbox{ if } 1\le p<2,\\
   q<\infty,& \mbox{ if } p=2,\\
   q\le \infty, & \mbox{ if } p>2.
   \end{array}
   \right.
\end{equation}
Then, there exists a constant $\Theta_{p,q}>0$ such that for every $\Omega \subseteq \mathbb{R}^2$ open multiply connected set of order $k \in \mathbb{N} \setminus \{0\}$ with finite inradius $r_\Omega$, we have  
\begin{equation} 
\label{makai-hayman-stima-2}
        \lambda_{p,q}(\Omega) \geq  \Theta_{p,q}\,\left(\frac{1}{\sqrt{k}\,r_{\Omega}}\right)^{p-2+\frac{2\,p}{q}}.
    \end{equation}
Moreover, the constant $\Theta_{p,q}$ has the following asymptotic behaviors: 
\begin{itemize}
\item for $1\le p<2$
\[
0<\lim_{q\nearrow p^*} \Theta_{p,q}<+\infty;
\]
\item for $p=2$
\[
0<\liminf_{q\nearrow \infty} \big(q\,\Theta_{2,q}\big)\le \limsup_{q\nearrow \infty} \big(q\,\Theta_{2,q}\big)<+\infty.
\]
\end{itemize}
\end{theorem}
\begin{proof}
We first prove inequality \eqref{makai-hayman-stima-2}, then by using the explicit expression of the constant $\Theta_{p,q}$, we will prove the second part of the statement.
\vskip.2cm\noindent
{\bf Part 1: inequality}.
   Up to a scaling, we can suppose that $r_\Omega = 1$. We take  $\delta = \lfloor \sqrt{k}\rfloor + 1 \in \mathbb{N}$ and consider the family of squares 
   \[
   Q_{ij} := Q_{5\delta}(10\, \delta\, i, 10\, \delta\, j), \qquad \mbox{ for every } (i,j) \in \mathbb{Z}^2.
   \]
We introduce the set of indices 
\[
\mathbb{Z}^2_\Omega = \{(i,j) \in \mathbb{Z}^2 \, :\, Q_{ij} \cap \Omega \neq \emptyset \},
\] 
and for every $(i,j)\in\mathbb{Z}^2_\Omega$ we take $\Sigma_{ij} \subseteq \overline{Q_{ij}}\setminus \Omega$ to be the compact set provided by Lemma \ref{taylor_lemma}. Let $u \in C^{\infty}_0(\Omega)$, then by Theorem \ref{thm:mazya_poincare} with $d=5\,\delta$ and $D = 2\,d=10\,\delta$, we have
   \[
\begin{split}
	\int_{\Omega} |\nabla u|^p dx &= \sum_{(i,j) \in \mathbb{Z}^2} \int_{Q_{ij}} |\nabla u|^p dx \geq \frac{\mathscr{C}^p}{(5\,\delta)^{\frac{2\,p}{q}}}\, \sum_{(i,j) \in \mathbb{Z}^2_{\Omega}} \mathrm{cap}_p(\Sigma_{ij}; \widetilde{B}_{ij})\, \|u\|_{L^q(Q_{ij})}^p,  
	\end{split}
\]
where we denoted with $\widetilde{B}_{ij}$ the ball with radius $D=2\,d=10\,\delta$, concentric with $Q_{ij}$. The key point now is to give a uniform bound from below on the capacity of the sets $\Sigma_{ij}$: by relying on Lemma \ref{cap_length} and Lemma \ref{taylor_lemma}, we can infer
\[
\mathrm{cap}_p(\Sigma_{ij};\widetilde{B}_{ij})\ge \frac{2}{(10\,\delta)^{p-1}}\,\max\Big\{\mathcal{H}^1(\Pi_{\mathbf{e}_1}(\Sigma_{ij})), \mathcal{H}^1(\Pi_{\mathbf{e}_2}(\Sigma_{ij}))\Big\}\ge \frac{\sqrt{k}}{2\cdot (10\,\delta)^{p-1}}.
\]
By collecting these estimates, we get
\[
\int_{\Omega} |\nabla u|^p dx\ge \frac{\mathscr{C}^p\,\sqrt{k}}{2^p\cdot (5\,\delta)^{p-1+\frac{2\,p}{q}}}\,\sum_{(i,j) \in \mathbb{Z}^2_{\Omega}} \|u\|_{L^q(Q_{ij})}^p.
\]
 and $\mathscr{C}$ is the same constant as in Theorem \ref{thm:mazya_poincare}. Since $k \geq 1$, we have 
\[
\frac{\sqrt{k}}{\delta^{p-1+\frac{2\,p}{q}}}=\frac{\sqrt{k}}{( \lfloor \sqrt{k}\rfloor + 1)^{p-1+\frac{2\,p}{q}}}\ge \frac{\sqrt{k}}{(2\,\sqrt{k})^{p-1+\frac{2\,p}{q}}}=\frac{1}{2^{p-1+\frac{2\,p}{q}}}\,\left(\frac{1}{\sqrt{k}}\right)^{p-2+\frac{2\,p}{q}}.
\]
In order to conclude the proof, we are only left to observe that $q\ge p$, thus the power $\tau\mapsto \tau^{p/q}$ is sub-additive. This entails that\footnote{In the limit case $q=\infty$, we just use that
\[
\sum_{(i,j) \in \mathbb{Z}^2_{\Omega}} \|u\|_{L^\infty(Q_{ij})}^p\ge \|u\|_{L^\infty(\Omega)}^p.
\]} 
\begin{equation}
\label{subadditive}
\sum_{(i,j) \in \mathbb{Z}^2_{\Omega}} \|u\|_{L^q(Q_{ij})}^p \geq \left(\sum_{(i,j) \in \mathbb{Z}^2_{\Omega}} \|u\|_{L^q(Q_{ij})}^q\right)^{\frac{p}{q}} = \|u\|_{L^q(\Omega)}^p,
\end{equation}
Then, we get 
\[
\int_{\Omega} |\nabla u|^p dx\ge  \frac{\mathscr{C}^{p}}{2^p \cdot 10^{p-1+\frac{2\,p}{q}}}\,\left(\frac{1}{\sqrt{k}}\right)^{p-2+\frac{2\,p}{q}}\, \|u\|_{L^q(\Omega)}^p,
\]
and \eqref{makai-hayman-stima-2} follows by definition of $\lambda_{p,q}(\Omega)$.
\vskip.2cm\noindent
{\bf Part 2: asymptotics for $\Theta_{p,q}$}. In {\bf Part 1} we have obtained the following constant
\[
\Theta_{p,q}=\frac{\mathscr{C}^{p}}{2^p \cdot 10^{p-1+\frac{2\,p}{q}}},
\]
with $\mathscr{C}$ as in Theorem \ref{thm:mazya_poincare}. Thus, in order to understand the asymptotic behaviour of $\Theta_{p,q}$ as $q$ goes to $p^*$ or to $\infty$, it is sufficient to focus on the same issue for the constant $\mathscr{C}^p$. By Remark \ref{remark:costante} and taking $N=2$, $d/D=1/2$, this is given by
\[
\mathscr{C}=\frac{1}{\alpha_{2,p}}\, \left(\frac{1}{2}\right)^{\frac{8}{p}+\frac{2}{q}}\,\left[\pi^\frac{1}{q}+\frac{4\,\pi^\frac{1}{p}}{1-\dfrac{\sqrt{2}}{2}}\,\left(\frac{1}{\mu_{p,q}(B_1)}\right)^\frac{1}{p}  \right]^{-1}.
\]
For $1\le p<2$, we have that (see \cite[Lemma 1.2]{BNTapp}) 
\[
\lim_{q\nearrow p^*} \mu_{p,q}(B_1)=\mu_{p,p^*}(B_1)>0.
\]
For a lower bound on the last constant, see for example \cite[Proposition 3.1]{CW}.
\par
The case $p=2$ is slightly more delicate. In this case, we have 
\[
\lim_{q\nearrow \infty} \mu_{2,q}(B_1)=0.
\]
More precisely, one can prove that
\[
4\,\pi\,e\le \liminf_{q\to \infty} \big(q\,\mu_{2,q}(B_1)\big)\le \limsup_{q\to \infty} \big(q\,\mu_{2,q}(B_1)\big)\le 8\,\pi\,e,
\]
see \cite[Proposition 1.5]{BNTapp}.
In light of the expression of $\mathscr{C}$, this is enough to conclude.
\end{proof}
\begin{remark}[Asymptotic optimality]
\label{rem:optimal}
Let $1\le p\le 2$ and let $p\le q$ satisfy \eqref{esponenti}. By proceeding as in \cite[Theorem 1.2, point (2)]{BB}, we can construct a sequence $\{\Omega_k\}_{k\in\mathbb{N}\setminus\{0\}}\subseteq \mathbb{R}^2$ of open sets such that $\Omega_k$ is multiply connected of order $k$
\[
r_{\Omega_k}\le C\qquad \mbox{ and }\qquad  \limsup_{k\to\infty} k^{\frac{p-2}{2}+\frac{2\,p}{q}}\,\lambda_{p,q}(\Omega_k)<+\infty.
\]
This shows that the lower bound \eqref{makai-hayman-stima-2} is sharp in its dependence on $k$, as $k$ goes to $\infty$. For $p>2$, we will see in the next section that this estimate can be considerably improved, by removing the dependence on $k$. 
\par
We also recall that for $1\le p<2$ we have
\[
\lim_{q\nearrow p^*} \lambda_{p,q}(\Omega)=\lambda_{p,p^*}(\Omega),
\]
and the latter is actually independent of the set $\Omega$: it simply coincides with the sharp constant in the Sobolev inequality for the whole space $\mathbb{R}^2$ (see for example \cite[Chapter I, Section 4.5]{St}). The asymptotic behaviour of the constant $\Theta_{p,q}$ in \eqref{makai-hayman-stima-2} is perfectly consistent with this fact.
\par
Finally, for $p=2$ we have that for a multiply connected planar set with finite inradius, it holds
\[
\lim_{q\nearrow \infty} q\,\lambda_{2,q}(\Omega)=8\,\pi\,e,
\]
see Corollary \ref{coro:MT} below. Thus, here as well, the asymptotic behaviour of the constant $\Theta_{2,q}$ is consistent with this limit.
\end{remark}

\begin{remark}[The case $1\le q<p$] 
\label{counterexample_brasco_ruffini}
We observe that the proof of Theorem \ref{thm:makai-hayman-multiply} does not for work for $q<p$: the main obstruction is the sub-additivity inequality \eqref{subadditive}. This is not a mere technicality:  
in the case $q < p$, {\it inequality \eqref{makai-hayman-stima-2} can not hold}. Indeed, it already fails for convex sets. The typical counter-example is given by the infinite strip $\Omega=\mathbb{R}\times(-1,1)$, for which we have 
\[
r_\Omega=1\qquad \mbox{ and }\qquad \lambda_{p,q}(\Omega)=0, \mbox{ for } 1\le q<p.	
\]
We refer to \cite[Proposition 6.1]{BrPini} for more details.
\end{remark}	
\subsection{Embeddings for homogeneous spaces}
In this subsection, we briefly discuss some consequences of Theorem \ref{thm:makai-hayman-multiply} on the embedding properties of the homogeneous Sobolev space $\mathscr{D}^{1,p}_0$. We recall that the latter is the completion of $C^\infty_0(\Omega)$, with respect to the norm
\[
\varphi\mapsto \|\nabla\varphi\|_{L^p(\Omega)},\qquad \mbox{ for every } \varphi\in C^\infty_0(\Omega).
\]
\begin{corollary}
\label{coro:immersione}
Let $k\in\mathbb{N}\setminus\{0\}$ and let $\Omega\subseteq\mathbb{R}^2$ be an open multiply connected set of order $k$. Let $1\le p\le 2$ and let $p\le q$ satisfy \eqref{esponenti}. Then we have
\[
\mathscr{D}^{1,p}_0(\Omega)\hookrightarrow L^q(\Omega)\qquad \Longleftrightarrow \qquad r_\Omega<+\infty.
\]
\end{corollary}
\begin{proof}
The validity of the continuous embedding $\mathscr{D}^{1,p}_0(\Omega)\hookrightarrow L^q(\Omega)$ is equivalent to the fact that $\lambda_{p,q}(\Omega)>0$. Thus, the implication $\Longleftarrow$ is a direct consequence of \eqref{makai-hayman-stima-2}. For the converse implication, it is sufficient to observe that for every disk $B_r(x_0)\subseteq \Omega$, we have 
\[
\lambda_{p,q}(\Omega)\le \lambda_{p,q}(B_r(x_0))=\frac{\lambda_{p,q}(B_1)}{r^{p-2+\frac{2\,p}{q}}}.
\]
By taking the supremum over the disks contained in $\Omega$, we get
\[
\lambda_{p,q}(\Omega)\le \frac{\lambda_{p,q}(B_1)}{r^{p-2+\frac{2\,p}{q}}_\Omega},
\]
and thus the conclusion.
\end{proof}
We now focus on the case $p=2$. In this case, there is no limit Sobolev exponent, i.e. the exponent $q$ may become arbitrary large, but it can not attain $\infty$. In general, the limit embedding for $\mathscr{D}^{1,2}_0(\Omega)$ is on the scale of Orlicz spaces of exponential type. For example, for open planar sets with finite area, the Moser-Trudinger inequality asserts that 
\[
\sup_{u\in C^\infty_0(\Omega)} \left\{\int_\Omega \big(\exp(4\,\pi\,u^2)-1\big)\,dx\, :\, \int_\Omega |\nabla u|^2\,dx= 1\right\}<+\infty,
\]
see \cite[Theorem 1]{Mos}.
In \cite[Theorem 1.2]{MS}, the authors proved that for an open simply connected set $\Omega\subseteq\mathbb{R}^2$, we have 
\[
\sup_{u\in C^\infty_0(\Omega)} \left\{\int_\Omega \big(\exp(4\,\pi\,u^2)-1\big)\,dx\, :\, \int_\Omega |\nabla u|^2\,dx= 1\right\}<+\infty\qquad \Longleftrightarrow\qquad r_\Omega<+\infty.
\]
In the next result, we extend this characterization to planar sets with non-trivial topology.
\begin{corollary}[Moser-Trundinger]
\label{coro:MT}
Let $k\in\mathbb{N}\setminus\{0\}$ and let $\Omega\subseteq\mathbb{R}^2$ be an open multiply connected set of order $k$. Then, we have
\[
\sup_{u\in C^\infty_0(\Omega)} \left\{\int_\Omega \big(\exp(4\,\pi\,u^2)-1\big)\,dx\, :\, \int_\Omega |\nabla u|^2\,dx= 1\right\}<+\infty\qquad\Longleftrightarrow \qquad r_\Omega<+\infty.
\]
Moreover, if $r_\Omega<+\infty$ we have 
\[
\lim_{q\nearrow \infty} q\,\lambda_{2,q}(\Omega)=8\,\pi\,e.
\]
\end{corollary}
\begin{proof}
According to \cite[Theorem 2.2]{BM}, for an open connected set $\Omega\subseteq\mathbb{R}^2$ we have that 
\[
\sup_{u\in C^\infty_0(\Omega)} \left\{\int_\Omega \big(\exp(4\,\pi\,u^2)-1\big)\,dx\, :\, \int_\Omega |\nabla u|^2\,dx= 1\right\}<+\infty\qquad\Longleftrightarrow \qquad \lambda(\Omega)<+\infty.
\]
If $\Omega$ is multiply connected of order $k$, the last condition is equivalent to $r_\Omega<+\infty$, thanks to Corollary \ref{coro:immersione} with $p=q=2$.
\par
The second statement now follows by reproducing verbatim the argument of \cite[Lemma 2.2]{RW}: the first part of the proof assures that we have the Moser-Trudinger inequality at our disposal, which is sufficient to reproduce the argument in \cite{RW}.
\end{proof}


\section{The case $p > N$}
\label{sec:4}

\subsection{Punctured Poincar\'e constants}
Let $p>N\ge 1$ and let $K\subseteq\mathbb{R}^N$ be an open bounded convex set. For every $x_0\in\overline K$, we define the following Poincar\'e constants
\[
\Lambda_p(K\setminus\{x_0\})=\inf_{u\in \mathrm{Lip}(\overline K)} \left\{\int_K |\nabla u|^p\,dx\, :\, \| u \|_{L^p(K)}=1,\, u(x_0)=0\right\},
\]
and
\[
\Lambda_{p, \infty}(K\setminus\{x_0\}) = \inf_{u\in \mathrm{Lip}(\overline K)} \left\{ \int_{K} |\nabla u|^p\, dx \, : \, \| u \|_{L^\infty(K)} = 1, u(x_0)=0 \right\}.
\]
We observe that in the particular case $K=B_R(x_0)$, we have
\begin{equation}
\label{scala}
\Lambda_p(B_R(x_0)\setminus\{x_0\})=\frac{\Lambda_p(B_1\setminus\{0\})}{R^p}\quad \mbox{ and }\quad \Lambda_{p,\infty}(B_R(x_0)\setminus\{x_0\})=\frac{\Lambda_{p,\infty}(B_1\setminus\{0\})}{R^{p-N}}.
\end{equation}
The following simple result is instrumental to get a lower bound on this constant.
\begin{lemma}
	\label{lm:lowerboundo}
	Let $1\le N<p$ and $R>0$. For every $u\in \mathrm{Lip}([0,R])\setminus\{0\}$ such that $u(0)=0$, we have
	\[
	\frac{\Lambda_p(B_1\setminus\{0\})}{R^p}\,\int_0^R |u(t)|^p\,t^{N-1}\,dt\le \int_0^R |u'(t)|^p\,t^{N-1}\,dt.
	\] 
\end{lemma}
\begin{proof}
	For every $u$ as in the statement, we define
	\[
	U(x)=u(|x-x_0|),\qquad \mbox{ for every } x\in B_R(x_0).
	\]
	By definition of $\Lambda_p(B_R(x_0);x_0)$, we have 
	\[
	\Lambda_p(B_R(x_0)\setminus\{x_0\})\,\int_{B_R(x_0)} |U|^p\,dx\le \int_{B_R(x_0)}|\nabla U|^p\,dx.
	\]
	By using spherical coordinates centered at $x_0$ and taking \eqref{scala} into account, we get the desired conclusion.
\end{proof}
We can thus prove the following sharp inequality, which is interesting in itself.
\begin{lemma} \label{lambdone_lemma}
	Let $1\le N<p$ and let $K\subseteq\mathbb{R}^N$ be an open bounded convex set. For every $x_0\in\overline K$, we have 
	\[
	\Lambda_p(K\setminus\{x_0\})\ge \frac{\Lambda_p(B_1\setminus\{0\})}{D_K(x_0)^p},\qquad \mbox{ where } D_K(x_0)=\max_{y\in\partial K} |x_0-y|.
	\]
	Moreover, we have equality for $K=B_{R}(x_0)$.
\end{lemma}
\begin{proof}
	Let $u$ be an admissible function for the problem which defines $\Lambda_p(K;\{x_0\})$. By using spherical coordinates centered at $x_0$, we get
	\[
	\begin{split}
		\int_K |\nabla u|^p\,dx&=\int_{\mathbb{S}^{N-1}}\int_0^{r(\omega)} \left[\left(\frac{\partial u}{\partial\varrho}\right)^2+\frac{1}{\varrho^2}\,|\nabla_\tau u|^2\right]^\frac{p}{2}\,\varrho^{N-1}\,d\varrho\,d\mathcal{H}^{N-1}(\omega) \\
		&\ge \int_{\mathbb{S}^{N-1}}\int_0^{r(\omega)} \left|\frac{\partial u}{\partial\varrho}\right|^p\,\varrho^{N-1}\,d\varrho\,d\mathcal{H}^{N-1}(\omega). 
	\end{split}
	\]
	By using Lemma \ref{lm:lowerboundo} in the innermost integral, we obtain
	\[
	\int_K |\nabla u|^p\,dx\ge  \Lambda_p(B_1\setminus\{0\})\,\int_{\mathbb{S}^{N-1}}\frac{1}{r(\omega)^p}\,\int_0^{r(\omega)} |u|^p\,\varrho^{N-1}\,d\varrho\,d\mathcal{H}^{N-1}(\omega).
	\]
	Finally, by noticing that 
	\[
	r(\omega)\le R_K(x_0),\qquad \mbox{ for every }\omega\in\mathbb{S}^{N-1},
	\]
	we get the desired conclusion.
\end{proof}

The following estimate on the quantities $\Lambda_p$ and $\Lambda_{p,\infty}$ will be useful in the sequel, in the particular case $K=B_1$ and $x_0=0$.

\begin{lemma}  
\label{lemma:Lambda_cap}
	Let $1\le N<p$ and let $K\subseteq\mathbb{R}^N$ be an open bounded convex set. For every $x_0\in\overline K$, we have the following estimates
	\begin{equation} 
	\label{porconi}
	|K|\,\Lambda_p(K\setminus\{x_0\}) \geq \Lambda_{p,\infty}(K\setminus\{x_0\})\ge \frac{\mu_{p,\infty}(K)}{2^p},
	\end{equation}
where we recall that $\mu_{p,\infty}(K)$ is defined in \eqref{muuu}.	
	Moreover, we have 
	\[
	\lim_{p\to\infty} \Big(\Lambda_p(K\setminus\{x_0\})\Big)^\frac{1}{p}=\lim_{p\to\infty} \Big(\Lambda_{p,\infty}(K\setminus\{x_0\})\Big)^\frac{1}{p}=\frac{1}{D_K(x_0)},
	\]
where as above $D_K(x_0)=\max_{x\in \partial K} |x-x_0|$.
\end{lemma}
\begin{proof}
The leftmost inequality in \eqref{porconi} easily follows from H\"older's inequality. In order to prove the rightmost one,
let $u$ be a Lipschitz function on $\overline{K}$, such that $u(x_0)=0$ and $\|u\|_{L^\infty(K)}=1$. Let $t_u$ be such that
\[
\|u-t_u\|_{L^\infty(K)}=\min_{t\in\mathbb{R}} \|u-t\|_{L^\infty(K)}.
\]
By definition of $\mu_{p,\infty}(K)$, we have 
\[
|u(x)-t_u|\le \left(\frac{1}{\mu_{p,\infty}(K)}\right)^\frac{1}{p}\,\left(\int_{K} |\nabla u|^p\,dx\right)^\frac{1}{p},\qquad \mbox{ for every } x\in K.
\]
Thus, we obtain
\[
\begin{split}
|u(x)|=|u(x)-u(x_0)|&\le |u(x)-t_u|+|u(x_0)-t_u|\\
&\le 2\,\left(\frac{1}{\mu_{p,\infty}(K)}\right)^\frac{1}{p}\,\left(\int_{K} |\nabla u|^p\,dx\right)^\frac{1}{p},\quad \mbox{ for every } x\in K.
\end{split}
\]
By taking the supremum over $x\in K$ and recalling the normalization on $u$, we get
\[
\frac{\mu_{p,\infty}(K)}{2^p}\le \Lambda_{p,\infty}(K\setminus\{0\}),
\]
as desired.
\vskip.2cm\noindent
For the second part of the statement, we first observe that if $N<p_1<p_2$, then
\[
\left(\frac{\Lambda_{p_1,\infty}(K\setminus\{x_0\})}{|K|}\right)^\frac{1}{p_1}\le \left(\frac{\Lambda_{p_2,\infty}(K\setminus\{x_0\})}{|K|}\right)^\frac{1}{p_2},
\]
by H\"older's inequality. Thus, the limit
\[
\lim_{p\to\infty} \left(\frac{\Lambda_{p,\infty}(K\setminus\{x_0\})}{|K|}\right)^\frac{1}{p},
\] 
exists, by monotonicity. This in turn implies that $\lim_{p\to\infty} (\Lambda_{p,\infty}(K\setminus\{x_0\}))^{1/p}$ exists, as well. In order to estimate this limit from above, we notice that the function 
\[
u(x)=\left(\int_K |x-x_0|^p\,dx\right)^{-\frac{1}{p}}\,|x-x_0|,
\]
is admissible for $\Lambda_p(K\setminus\{x_0\})$. Thus, we get
\begin{equation}
\label{limsup}
\begin{split}
\lim_{p\nearrow\infty} \Big(\Lambda_{p,\infty}(K\setminus\{x_0\})\Big)^\frac{1}{p}&\le \lim_{p\nearrow \infty} \Big(\Lambda_p(K\setminus\{x_0\})\Big)^\frac{1}{p}\\
&\le \lim_{p\nearrow \infty} |K|^\frac{1}{p}\,\left(\int_K |x-x_0|^p\,dx\right)^{-\frac{1}{p}}=\frac{1}{D_K(x_0)}.
\end{split}
\end{equation}
Observe that we used that 
\[
D_K(x_0)=\max_{x\in \partial K} |x-x_0|=\max_{x\in \overline K} |x-x_0|,
\]
thanks to the convexity of $K$.
\par
The estimate from below is more elaborated, but the argument is nowadays quite standard (see for example \cite[Section 2]{BdBM}). For every $m\in \mathbb{N}$ such that $m\ge N+1$, let us take $u_{m}\in \mathrm{Lip}(\overline K)$ such that
\[
\|u_{m}\|_{L^\infty(K)}=1, \qquad u_{m}(x_0)=0,\qquad\int_K |\nabla u_{m}|^m\,dx<2\,\Lambda_{m,\infty}(K\setminus\{x_0\}).
\]
By H\"older's inequality, for every $m\ge N+1$ we have
\[
\int_K |\nabla u_m|^{N+1}\,dx\le |K|^{1-\frac{N+1}{m}}\,\left(\int_K |\nabla u_m|^m\,dx\right)^\frac{N+1}{m}\le |K|^{1-\frac{N+1}{m}}\,\Big(2\,\Lambda_{m,\infty}(K\setminus\{x_0\})\Big)^\frac{N+1}{m}.
\]
In light of \eqref{limsup}, this shows that $\{u_m\}_{m\ge N+1}$ is a bounded sequence in $W^{1,N+1}(K)$. By the Morrey-Sobolev compact embedding (see \cite[Theorem 12.61]{Leonibook}), we have that there exists a subsequence $\{u_{m^1_n}\}_{n\in \mathbb{N}}\subseteq \{u_m\}_{m\ge N+1}$ and a limit function $u_{\infty}\in W^{1,N+1}(K)\cap C(\overline{K})$, such that $u_{m^1_n}$ converges weakly in $W^{1,N+1}(K)$ and uniformly on $\overline{K}$ to $u_{\infty}$. Thus, we still have
\[
\|u_{\infty}\|_{L^\infty(K)}=1, \qquad u_{\infty}(x_0)=0.
\]
Moreover, by lower semicontinuity and \eqref{limsup}, we have 
\[
\left(\int_K |\nabla u_{\infty}|^{N+1}\,dx\right)^\frac{1}{N+1}\le \frac{|K|^\frac{1}{N+1}}{D_K(x_0)}.
\]
We can now recursively repeat the previous argument: we take $N+\ell+1$ for $\ell\in\mathbb{N}\setminus\{0\}$ and extract a subsequence $\{u_{m^{\ell+1}_n}\}_{n\in\mathbb{N}}$ from the previous one $\{u_{m^{\ell}_n}\}_{n\in\mathbb{N}}$. Indeed, at each step, we have 
\begin{equation}
\label{}
\int_K |\nabla u_{m_n^\ell}|^{N+\ell+1}\,dx\le |K|^{1-\frac{N+\ell+1}{m_n^\ell}}\,\Big(2\,\Lambda_{m_n^\ell,\infty}(K\setminus\{x_0\})\Big)^\frac{N+\ell+1}{m_n^\ell},
\end{equation}
which shows that $\{u_{m^\ell_n}\}_{n\in\mathbb{N}}$ is a bounded sequence in $W^{1,N+\ell+1}(K)$. As before, there exists a subsequence $\{u_{m^{\ell+1}_n}\}_{n\in\mathbb{N}}$ which converges weakly in $W^{1,N+\ell+1}(K)$ and uniformly on $\overline{K}$. By construction, the limit function must still coincide with the original limit function $u_{\infty}$. This shows that $u_{\infty}\in W^{1,N+\ell+1}$ for every $\ell\in\mathbb{N}$ and that
\begin{equation}
\label{rottoilca}
\begin{split}
\left(\int_K |\nabla u_{\infty}|^{N+\ell+1}\,dx\right)^\frac{1}{N+\ell+1}&\le \lim_{n\to\infty}|K|^{\frac{1}{N+\ell+1}-\frac{1}{m_n^\ell}}\,\Big(2\,\Lambda_{m_n^\ell,\infty}(K\setminus\{x_0\})\Big)^\frac{1}{m_n^\ell}\\
&\le\frac{|K|^\frac{1}{N+\ell+1}}{D_K(x_0)}.
\end{split}
\end{equation}
By taking the limit as $\ell$ goes to $\infty$, we get that $u_\infty\in \mathrm{Lip}(\overline{K})$, with 
\[
\|\nabla u_\infty\|_{L^\infty(K)}\le \frac{1}{D_K(x_0)},\qquad \|u_\infty\|_{L^\infty(K)}=1, \qquad u_\infty(x_0)=0.
\]
Actually, the last two properties show that the first one can be improved.
Indeed, let $\overline{x}\in\overline{K}$ be a maximum point of $|u_\infty|$. We then have\footnote{As already observed, by convexity of $K$, we have
\[
D_K(x_0)=\max_{x\in \partial K} |x-x_0|=\max_{x\in \overline{K}} |x-x_0|.
\]} 
\[
1=|u_\infty(\overline{x})|=|u_\infty(\overline{x})-u_\infty(x_0)|\le \|\nabla u_\infty\|_{L^\infty(K)}\,|\overline{x}-x_0|\le \frac{|\overline{x}-x_0|}{D_K(x_0)}\le 1.
\]
This implies that equality must hold everywhere. In particular, we get 
\begin{equation}
\label{scassatoilca}
\|\nabla u_\infty\|_{L^\infty(K)}= \frac{1}{D_K(x_0)}.
\end{equation}
With this information at hand, we can now conclude: we go back to \eqref{rottoilca} and observe that 
\[
\begin{split}
\lim_{n\to\infty}|K|^{\frac{1}{N+\ell+1}-\frac{1}{m_n^\ell}}\,\Big(2\,\Lambda_{m_n^\ell,\infty}(K\setminus\{0\})\Big)^\frac{1}{m_n^\ell}
&=|K|^\frac{1}{N+\ell+1}\,\lim_{m\to\infty}\Big(\Lambda_{m,\infty}(K\setminus\{0\})\Big)^\frac{1}{m}.
\end{split}
\]
Thus, we obtain
\[
\lim_{m\to\infty}\Big(\Lambda_{m,\infty}(K\setminus\{0\})\Big)^\frac{1}{m}\ge |K|^{-\frac{1}{N+\ell+1}}\,\left(\int_K |\nabla u_\infty|^{N+\ell+1}\,dx\right)^\frac{1}{N+\ell+1}.
\]
By taking the limit as $\ell$ goes to $\infty$ and using \eqref{scassatoilca}, we conclude.
\end{proof}
We conclude this part, by observing that $\Lambda_p(Q_1\setminus\{0\})$ actually coincides with $\lambda_p$ of a suitable ``pepper'' set. More precisely, we have the following
\begin{lemma}
\label{lemma:forelli}
For $1\le N<p$, we have
\[
\lambda_{p}(\mathbb{R}^N\setminus \mathbb{Z}^N)=\Lambda_{p}(Q_{1/2}\setminus\{0\})=2^p\,\Lambda_{p}(Q_1\setminus\{0\}).
\]
\end{lemma}
\begin{proof}
The rightmost equality simply follows by scaling. Let us prove the leftmost one.
\par
Let $u\in C^\infty_0(\mathbb{R}^N\setminus \mathbb{Z}^N)$. By tiling the space with the cubes
\[
Q_{1/2}(\mathbf{i}),\qquad \mbox{ with }\mathbf{i}\in\mathbb{Z}^N,
\]
we easily see that $u$ is admissible for the variational problem which defines $\Lambda_p(Q_{1/2}(\mathbf{i})\setminus\{\mathbf{i}\})$. Thus, we get
\[
\int_\Omega |\nabla u|^p\,dx=\sum_{\mathbf{i}\in\mathbb{Z}^N} \int_{Q_{1/2}(\mathbf{i})} |\nabla u|^p\,dx\ge \sum_{\mathbf{i}\in\mathbb{Z}^N} \Lambda_p(Q_{1/2}(\mathbf{i})\setminus\{\mathbf{i}\})\,\int_{Q_{1/2}(\mathbf{i})} |u|^p\,dx.
\]
Since we have 
\[
\Lambda_p(Q_{1/2}(\mathbf{i})\setminus\{\mathbf{i}\})=\Lambda_p(Q_{1/2}\setminus\{0\}),\qquad \mbox{ for every } \mathbf{i}\in\mathbb{Z}^N,
\]
we can infer 
\[
\int_\Omega |\nabla u|^p\,dx\ge \Lambda_p(Q_{1/2}\setminus\{0\})\,\int_{\Omega} |u|^p\,dx.
\]
By the arbitrariness of $u$, this yields $\lambda_{p}(\mathbb{R}^N\setminus \mathbb{Z}^N)\ge \Lambda_{p}(Q_{1/2}\setminus\{0\})$. 
\par
In order to prove the reverse inequality, we take $u\in \mathrm{Lip}(\overline{Q_{1/2}})$ such that $u(0)=0$ and $\|u\|_{L^p(Q_{1/2})}=1$. According to Lemma \ref{lemma:cube_symmetry}, we can further suppose that $u$ is non-negative and symmetric with respect to each variable.
For every $m\in\mathbb{N}$, we define
\[
\mathbb{Z}^N_m=\Big\{\mathbf{i}=(i_1,\dots,i_N)\in\mathbb{Z}^N\, :\, |\mathbf{i}|_{\ell^\infty}:=\max_{k=1,\dots,N}|i_k|\le m\Big\},
\]
and then we set
\[
U_m(x)=\sum_{\mathbf{i}\in\mathbb{Z}^N_m} u(x+\mathbf{i}).
\] 
Observe that this is Lipschitz function on the cube $Q_{m+1/2}$ (thanks to the symmetries of $u$), vanishing at each point $\mathbf{i}\in \mathbb{Z}^N_m$. We then take $\eta_m$ a $1$-Lipschitz cut-off function, such that
\[
0\le \eta_m\le 1,\qquad \eta_m\equiv 1 \mbox{ on }  Q_{m-1/2},\qquad \eta_m=0 \mbox{ on } \partial Q_{m+1/2}.
\]
By construction, we get that $\eta_m\,U_m\in W^{1,p}_0(\mathbb{R}^N\setminus \mathbb{Z}^N)$, where we extend it by $0$ outside $Q_{m+1/2}$.
Thus, we get
\[
\begin{split}
\Big(\lambda_{p}(\mathbb{R}^N\setminus \mathbb{Z}^N)\Big)^\frac{1}{p}&\le \frac{\left(\displaystyle \int_{\mathbb{R}^N} |\nabla (\eta_m\,U_m)|^p\,dx\right)^\frac{1}{p}}{\displaystyle \left(\int_{\mathbb{R}^N}|\eta_m\,U_m|^p\,dx\right)^\frac{1}{p}}\\
&\le \frac{\left(\displaystyle \sum_{\mathbf{i}\in\mathbb{Z}^N_m}\int_{Q_{1/2}(\mathbf{i})} |\nabla \eta_m|^p\,|U_m|^p\,dx\right)^\frac{1}{p}}{\displaystyle \left(\sum_{\mathbf{i}\in\mathbb{Z}^N_m}\int_{Q_{1/2}(\mathbf{i})}|\eta_m\,U_m|^p\,dx\right)^\frac{1}{p}}+\frac{\left(\displaystyle \sum_{\mathbf{i}\in\mathbb{Z}^N_m}\int_{Q_{1/2}(\mathbf{i})} |\nabla U_m|^p\,|\eta_m|^p\,dx\right)^\frac{1}{p}}{\displaystyle \left(\sum_{\mathbf{i}\in\mathbb{Z}^N_m}\int_{Q_{1/2}(\mathbf{i})}|\eta_m\,U_m|^p\,dx\right)^\frac{1}{p}}.
\end{split}
\]
We now observe that, thanks to the properties of $\eta_m$, we have
\[
\begin{split}
\sum_{\mathbf{i}\in\mathbb{Z}^N_m}\int_{Q_{1/2}(\mathbf{i})}|\eta_m\,U_m|^p\,dx\ge \sum_{\mathbf{i}\in\mathbb{Z}^N_{m-1}}\int_{Q_{1/2}(\mathbf{i})}|\eta_m\,U_m|^p\,dx&=\sum_{\mathbf{i}\in\mathbb{Z}^N_{m-1}}\int_{Q_{1/2}(\mathbf{i})}|U_m|^p\,dx\\
&=(2\,m-1)^N\,\int_{Q_{1/2}}|u|^p\,dx.
\end{split}
\]
We also used that $U_m$ coincides with a translated copy of the original function $u$ defined on $Q_{1/2}$. As for the first integral at the numerator, since $\eta_m$ is $1-$Lipschitz and is constant on $Q_{m-1/2}$, we get
\[
\sum_{\mathbf{i}\in\mathbb{Z}^N_m}\int_{Q_{1/2}(\mathbf{i})} |\nabla \eta_m|^p\,|U_m|^p\,dx\le \sum_{|\mathbf{i}|_{\ell^\infty}=m}\int_{Q_{1/2}(\mathbf{i})} |U_m|^p\,dx=\Big[(2\,m+1)^N-(2\,m-1)^N\Big]\,\int_{Q_{1/2}} |u|^p\,dx.
\]
Finally, by using that $|\eta_m|\le 1$, we have 
\[
\sum_{\mathbf{i}\in\mathbb{Z}^N_m}\int_{Q_{1/2}(\mathbf{i})} |\nabla U_m|^p\,|\eta_m|^p\,dx\le \sum_{\mathbf{i}\in\mathbb{Z}^N_m}\int_{Q_{1/2}(\mathbf{i})} |\nabla U_m|^p\,dx=(2\,m+1)^N\,\int_{Q_{1/2}} |\nabla u|^p\,dx.
\]
By using these estimates, we get
\[
\Big(\lambda_{p}(\mathbb{R}^N\setminus \mathbb{Z}^N)\Big)^\frac{1}{p}\le \left(\left(\frac{2\,m+1}{2\,m-1}\right)^N-1\right)^\frac{1}{p}+\left(\frac{2\,m+1}{2\,m-1}\right)^\frac{N}{p}\frac{\|\nabla u\|_{L^p(Q_{1/2})}}{\|u\|_{L^p(Q_{1/2})}}.
\]
If we now take the limit as $m$ goes to $\infty$, this yields
\[
\Big(\lambda_{p}(\mathbb{R}^N\setminus \mathbb{Z}^N)\Big)^\frac{1}{p}\le \frac{\|\nabla u\|_{L^p(Q_{1/2})}}{\|u\|_{L^p(Q_{1/2})}}.
\]
Since $u$ is arbitrary, we get $\lambda_p(\mathbb{R}^N\setminus\mathbb{Z}^N)\le \Lambda_p(Q_1\setminus\{0\})$, as well.
\end{proof}

\subsection{Inradius and principal frequencies}
By using the punctured Poincar\'e constants of the previous subsection, we will now derive a lower bound on $\lambda_{p,q}$ for $p>N$ in terms of the inradius, which is valid for every open set.
This generalizes \cite[Theorem 1.4.1]{Po1}, by means of a different proof. Moreover, we pay due attention to the quality of the constant obtained. In particular, we prove that it has the correct asymptotic behaviour, in the regimes $p\searrow N$ and $p\nearrow \infty$.
\begin{theorem}[Endpoint cases]
\label{thm:endpoint}
	Let $1\le N < p$. Then, for every $\Omega \subseteq \mathbb{R}^N$ open set with finite inradius $r_{\Omega}$, we have 
	\begin{equation} 
	\label{makai_super_case}
		\lambda_p(\Omega) \geq \frac{\beta_{N,p}}{r_{\Omega}^p},\qquad \mbox{ with }\beta_{N,p} = \max \left\{ \frac{ \Lambda_p(B_1\setminus\{0\})}{(\sqrt{N} + 1)^p} , \left(\frac{p-N}{p}\right)^p \right\}>0,
	\end{equation}
	and
		\begin{equation} 
	\label{lambda_p_inf}
 		\lambda_{p, \infty}(\Omega) \geq \frac{\Lambda_{p, \infty}(B_1\setminus\{0\})}{r_{\Omega}^{p-N}}.
 	\end{equation}
Moreover, the constants $\beta_{N,p}$ and $\Lambda_{p,\infty}(B_1\setminus\{0\})$ have the following asymptotic behaviours  
	\[
	0<\liminf_{p\searrow N} \frac{\beta_{N,p}}{(p-N)^{p-1}}\le \limsup_{p\searrow N} \frac{\beta_{N,p}}{(p-N)^{p-1}}<+\infty,
	\]
	\[
0<\liminf_{p\searrow N} \frac{\Lambda_{p, \infty}(B_1; 0)}{(p-N)^{p-1}}\le \limsup_{p\searrow N} \frac{\Lambda_{p, \infty}(B_1; 0)}{(p-N)^{p-1}}<+\infty,
\]
and
	 \[
	\lim_{p\to\infty} \Big(\beta_{N,p}\Big)^\frac{1}{p}=\lim_{p\to\infty} \Big(\Lambda_{p, \infty}(B_1\setminus \{0\})\Big)^\frac{1}{p}=1.
	\]
\end{theorem}
\begin{proof} We divide the proof in four parts.
\vskip.2cm\noindent
{\bf Part 1: inequality for $q=p$}. 
From \cite[Theorem 5.4 \& Remark 5.5]{BPZ2}, we already have
\begin{equation}
\label{mezza2_stima_super_case}
\lambda_p(\Omega)\ge \left(\frac{p-N}{p}\right)^p\,\frac{1}{r_\Omega^p}.
\end{equation}
This is a plain consequence of the Hardy inequality contained in \cite[Theorem 1.1]{GPP}. 
Unfortunately, the constant obtained in this way has a sub-optimal behaviour as $p\searrow N$. In order to rectify this fact, we give a different proof, based on Taylor's idea of tiling the space with cubes ``large enough''. We will see that for $p>N$, the situation is simpler.
\par
Without loss of generalization, we can assume $r_{\Omega} = 1$. We fix $\varepsilon>0$ and consider the tiling of $\mathbb{R}^N$ made by the cubes 
\[
\mathcal{Q}_{\mathbf{i},\varepsilon} := Q_{1+\varepsilon}((2+2\,\varepsilon)\,\mathbf{i}), \qquad \mbox{ for } \mathbf{i} \in \mathbb{Z}^N.
\]
We also consider the set of indices 
\[
\mathbb{Z}^N_{\Omega,\varepsilon} = \Big\{\mathbf{i}=(i_1, \ldots, i_N) \in \mathbb{Z}^N \,: \, \mathcal{Q}_{\mathbf{i},\varepsilon} \cap \Omega \neq \emptyset \Big\}.
\]
Let $u \in C_0^{\infty}(\Omega)$, we observe that for every $\mathbf{i} \in \mathbb{Z}^N_{\Omega,n}$ there must exist 
\[
x_{\mathbf{i},\varepsilon}\in B_{1+\varepsilon}((2+2\,\varepsilon)\,\mathbf{i})\setminus \Omega,
\]
thanks to the fact that $r_\Omega=1$: this implies that a ball of radius $1+\varepsilon$ can not be entirely contained in $\Omega$.
Then, by the tiling property of the collection $\{\mathcal{Q}_{\mathbf{i},\varepsilon}\}_{i \in \mathbb{Z}^N}$, the definition of $\Lambda_p(\mathcal{Q}_{\mathbf{i},\varepsilon}\setminus\{x_{\mathbf{i},\varepsilon}\})$ and Lemma \ref{lambdone_lemma} applied to each cube of this collection,  we get
\[
\begin{split}
		\int_{\Omega} |\nabla u|^p dx &= \sum_{\mathbf{i}\in \mathbb{Z}^N_{\Omega, \varepsilon}} \int_{\mathcal{Q}_{\mathbf{i},\varepsilon}} |\nabla u|^p dx \\  &\geq  \sum_{\mathbf{i}\in \mathbb{Z}^N_{\Omega, \varepsilon}} \Lambda_p(\mathcal{Q}_{\mathbf{i},\varepsilon}\setminus\{x_{\mathbf{i},\varepsilon}\})\, \|u\|^p_{L^p(\mathcal{Q}_{\mathbf{i},\varepsilon})} \geq \frac{\Lambda_p(B_1\setminus\{0\})}{\left(\left(1+\varepsilon\right)\,\sqrt{N} + 1 + \varepsilon\right)^p}\, \|u\|^p_{L^p(\Omega)}.
	\end{split}
	\]
Observe that we used that $x_{\mathbf{i},\varepsilon}\in B_{1+\varepsilon}((2+2\,\varepsilon)\,\mathbf{i})$, to infer that	
\[
\max_{y\in\partial \mathcal{Q}_{\mathbf{i},\varepsilon}} |x_0-y|\le \left(1+\varepsilon\right)\,\sqrt{N} + 1 + \varepsilon.
\]	
By taking the limit as $\varepsilon$ goes to $0$ in the estimate above, we obtain 
	\begin{equation} 
	\label{mezza_stima_super_case}
			\|\nabla u\|^p_{L^p(\Omega)}  \geq \frac{\Lambda_p(B_1\setminus\{0\})}{(\sqrt{N} + 1)^p}\, \|u\|^p_{L^p(\Omega)}. 
	\end{equation}
This is enough to obtain the lower bound on $\lambda_p(\Omega)$. 
Finally, by joining the two estimates \eqref{mezza_stima_super_case} and \eqref{mezza2_stima_super_case} we obtain \eqref{makai_super_case}, with the claimed constant $\beta_{N,p}$. 
\vskip.2cm\noindent
{\bf Part 2: inequality for $q=\infty$.} We can now give the counterpart of \eqref{makai_super_case}, for the endpoint case $q=\infty$. The argument is extremely simple, based on the properties of the $L^\infty$ norm and on a basic geometric fact.  As before, up to scaling, we can assume that $r_{\Omega} = 1$.  For every $\varepsilon > 0$, we consider the family of balls 
\[
\Big\{B_{1+\varepsilon}(y) \, : \, y \in \partial \Omega\Big\}.
\] 
It is not difficult to see that this is a covering of $\Omega$. Indeed, by definition of inradius, for every $x \in \Omega$, there exists $y\in \partial \Omega$ such that 
\[
d_\Omega(x)=|x-y|\le r_\Omega=1.
\]
In particular, this implies that $x \in B_{1+\varepsilon}(y)$. By arbitrariness of $x\in\Omega$, we get the claimed covering property.
\par
We now take $u \in C^{\infty}_0(\Omega)$. In particular, this is a continuous compactly supported function. Hence, there exists $\overline{x} \in \Omega$ such that 
\[
|u(\overline{x})|=\|u\|_{L^\infty(\Omega)}.
\] 
Thanks to the previous discussion, there exists $\overline{y} \in \partial \Omega$ such that $\overline{x} \in B_{1+\varepsilon}(\overline{y})$. Thus, we obtain 
\[
\begin{split}
 	\| u \|^p_{L^\infty(\Omega)} =|u(\overline{x})|^p&= \|u\|^p_{L^\infty(B_{1+\varepsilon}(\overline{y}))} \\ &\leq \frac{1}{\Lambda_{p, \infty}(B_{1+\varepsilon}(\overline{y}), \overline{y})} \int_{\Omega} |\nabla u|^p dx = \frac{(1+\varepsilon)^{p-N}}{\Lambda_{p, \infty}(B_1; 0)} \int_{\Omega} |\nabla u|^p dx.
 	\end{split}
	\]
In the last equality we used \eqref{scala}.	
By letting $\varepsilon$ go to $0$, we obtain \eqref{lambda_p_inf}.
\vskip.2cm\noindent
{\bf Part 3: asymptotics for $\Lambda_{p,\infty}(B_1\setminus\{0\})$.} By Lemma \ref{lemma:Lambda_cap}, we know that 
\[
\Lambda_{p,\infty}(B_1\setminus\{0\})\ge \frac{\mu_{p,\infty}(B_1)}{2^p}.
\]
In turn, the right-hand side can be bounded from below thanks to \eqref{lowerbound_mu}. Thus, we obtain
\[
\Lambda_{p,\infty}(B_1\setminus\{0\})\ge \frac{N}{2^p}\,\left(\frac{\omega_N}{2^N}\right)^p\, \left(\frac{p-N}{p-1}\right)^{p-1}\, \omega_N,
\]
which gives the claimed asymptotic behaviour from below, as $p$ goes to $N$.
\par
On the other hand, by testing the definition of $\Lambda_{p, \infty}(B_1\setminus\{0\})$ with 
\begin{equation}
\label{smussala!}
u_\varepsilon(x):= \Big(\varepsilon^2+|x|^2\Big)^{\frac{p-N}{2\,(p-1)}}-\varepsilon^\frac{p-N}{p-1},\qquad \mbox{ with } \varepsilon>0,
\end{equation}
we get 
\[
\Lambda_{p, \infty}(B_1\setminus\{0\}) \leq \lim_{\varepsilon\searrow 0}\frac{\displaystyle\int_{B_1} |\nabla u_\varepsilon|^p\,dx}{\|u_\varepsilon\|_{L^\infty(B_1)}^p}=N\,\omega_N\,\left(\frac{p-N}{p-1}\right)^{p-1}.
\]
This gives the desired asymptotic behaviour from above, as well. Finally, for the limit $p\nearrow \infty$ it is sufficient to use Lemma \ref{lemma:Lambda_cap} with $K=B_1$ and $x_0=0$.
\vskip.2cm\noindent
{\bf Part 4: asymptotics for $\beta_{N,p}$.} We recall that this is given by
\[
	\beta_{N,p} =  \max \left\{ \frac{ \Lambda_p(B_1\setminus\{0\})}{(\sqrt{N} + 1)^p} , \left(\frac{p-N}{p}\right)^p \right\}.
	\]
	In particular, by Lemma \ref{lemma:Lambda_cap} we have 
\[			
\beta_{N,p} \geq \max \left\{ \frac{\Lambda_{p,\infty}(B_1\setminus\{0\})}{\omega_N\,(\sqrt{N} + 1)^p} , \left(\frac{p-N}{p}\right)^{p}\right\}.
\]
Thus, the information
\[
0<\liminf_{p\searrow N} \frac{\beta_{N,p}}{(p-N)^{p-1}},
\]	
comes from {\bf Part 3}. The related upper bound can be proved as before, by using \eqref{smussala!} as a test function and taking the limit as $\varepsilon$ goes to $0$.
\par
Finally, as for the limit $p\nearrow \infty$, we observe that by its definition
\[
\liminf_{p\nearrow \infty} \Big(\beta_{N,p}\Big)^\frac{1}{p}\ge \liminf_{p\nearrow \infty} \frac{p-N}{p}=1.
\]
On the other hand, by using \eqref{makai_super_case} with $\Omega=B_1$, we get
\[
\limsup_{p\nearrow \infty} \Big(\beta_{N,p}\Big)^\frac{1}{p}\le \lim_{p\nearrow \infty} \Big(\lambda_{p}(B_1)\Big)^\frac{1}{p}=1,
\]
thanks to \cite[Lemma 1.5]{JLM}.
This concludes the proof.
\end{proof}

\begin{remark}[Asymptotic optimality]
\label{rem:asympN}
We recall that for {\it every} open set $\Omega\subseteq\mathbb{R}^N$, we have 
\[
\lim_{p\to\infty} \Big(\lambda_{p}(\Omega)\Big)^\frac{1}{p}=\lim_{p\to\infty} \Big(\lambda_{p,\infty}(\Omega)\Big)^\frac{1}{p}=\frac{1}{r_{\Omega}},
\]
see \cite[Corollary 6.1 and Corollary 6.4]{BPZ2}.
Thus, the estimates \eqref{makai_super_case} and \eqref{lambda_p_inf} becomes identities in the limit as $p$ goes to $\infty$. 
\par
As for the case when $p$ goes to $N$: from Lemma \ref{lemma:forelli} we know that
\[
\lambda_{p}(\mathbb{R}^N\setminus \mathbb{Z}^N)=\Lambda_{p}(Q_{1/2}\setminus\{0\}).
\]
The last quantity can be estimated from above by using the test function \eqref{smussala!}, as before. This gives
\[
\limsup_{p\searrow N} \frac{\lambda_{p}(\mathbb{R}^N\setminus \mathbb{Z}^N)}{(p-N)^{p-1}}<+\infty,
\]
and thus the constant $\beta_{N,p}$ in \eqref{makai_super_case} vanishes with the sharp decay rate. Finally, from both the definition of $\lambda_{p,\infty}$ and that of $p-$capacity, we have 
\[
\lambda_{p,\infty}(B_1)\le \mathrm{cap}_p(\{0\};B_1)=N\,\omega_N\,\left(\frac{p-N}{p-1}\right)^{p-1} .
\]
Thus, also the constant in \eqref{lambda_p_inf} has the sharp decay rate to $0$.
\end{remark}

By a simple interpolation argument, we can now get a geometric lower bound for the generalized principal frequencies in the whole range $p\le q \le \infty$, whenever $p > N$. Thanks to the previous discussion, this comes with a constant having a good behaviour in both asymptotics regime $p\searrow N$ and $p\nearrow \infty$.
 \begin{corollary}[Interpolation]
 \label{coro:interpol}
 Let $1\le N<p\le q\le \infty$ and let $\Omega \subseteq \mathbb{R}^N$ be an open set, with finite inradius $r_\Omega$. Then 
 \begin{equation}
 		\lambda_{p, q}(\Omega) \geq \Big(\beta_{N,p}\Big)^{\frac{p}{q}}\, \Big(\Lambda_{p, \infty}(B_1\setminus\{0\})\Big)^{1-\frac{p}{q}}\, \left(\frac{1}{r_{\Omega}}\right)^{p-N+N\frac{p}{q}},
 	\end{equation}
	where $\beta_{N,p}$ is the same as in \eqref{makai_super_case}.
 \end{corollary}
 \begin{proof}
 	For every $u \in C^{\infty}_0(\Omega) \backslash \{0\}$ we have \[
 	\| u\|_{L^q(\Omega)}^p \leq \left( \| u\|_{L^{\infty}(\Omega)}^{1-\frac{p}{q}} \right)^p \left(\| u\|_{L^p(\Omega)}^{\frac{p}{q}} \right)^p,
 	\]
 	then
 	\[
 	\frac{\|\nabla u\|^p_{L^p(\Omega)}}{\|u\|^p_{L^q(\Omega)}} \geq \left(\frac{\|\nabla u\|^p_{L^p(\Omega)}}{\| u\|^p_{L^{\infty}(\Omega)}}\right)^{1-\frac{p}{q}} \left(\frac{\|\nabla u\|^p_{L^p(\Omega)}}{\| u\|^p_{L^{p}(\Omega)}}\right)^{\frac{p}{q}}.
 	\]
	This entails that
 	 \[
 	\lambda_{p,q}(\Omega) \geq \Big(\lambda_{p, \infty}(\Omega)\Big)^{1-\frac{p}{q}}\, \Big(\lambda_p(\Omega)\Big)^{\frac{p}{q}}.
 	\]
 	Hence, the thesis follows by combining \eqref{makai_super_case} and \eqref{lambda_p_inf}.
	 \end{proof}

\section{Cheeger's constant and Buser's inequality}
\label{sec:5}
We recall the definition of {\it Cheeger constant} for an open set $\Omega\subseteq\mathbb{R}^N$. This is given by 
\[
h(\Omega)=\inf\left\{\frac{\mathcal{H}^{N-1}(\partial E)}{|E|}\, :\, E\Subset\Omega \mbox{ has a smooth boundary}\right\}.
\]
We point out that other definitions are possible, see for example the survey papers \cite{Leo} and \cite{Pa2}. The above definition is in the spirit of the original analogous quantity introduced by Cheeger in \cite{cheeger} (and especially by Buser, see \cite[equation (1.5)]{Bu2}) in the context of Riemannian manifolds. 
\par
Our choice is motivated by the following result, which is quite classical. Its proof is based on the Coarea Formula, see for example \cite[Theorem 2.3.1]{Maz}. It asserts that the geometric constant $h(\Omega)$ coincides with a generalized principal frequency.
\begin{lemma}
\label{lemma:cheegerclassic}
Let $N\ge 1$ and let $\Omega\subseteq\mathbb{R}^N$ be an open set. Then 
\[
\lambda_{1,1}(\Omega)=h(\Omega).
\]
\end{lemma}
By combining Lemma \ref{lemma:cheegerclassic} and Theorem \ref{thm:makai-hayman-multiply}, we immediately get the following lower bound on the Cheeger constant of a set, in terms of both its inradius and topology.
\begin{corollary}
\label{coro:cheeger}
Let $k\in\mathbb{N}\setminus\{0\}$, for every $\Omega \subseteq \mathbb{R}^2$ open multiply connected set of order $k$ with finite inradius $r_\Omega$, we have 
\begin{equation} 
\label{cheeger_lower}
        h(\Omega) \geq  \frac{\Theta_{1,1}}{\sqrt{k}}\,\frac{1}{r_{\Omega}},
    \end{equation}
    where $\Theta_{1,1}$ is the same constant as in Theorem \ref{thm:makai-hayman-multiply}.
\end{corollary}
\begin{remark}
It is easily seen that the geometric lower bound \eqref{cheeger_lower} {\it is not possible} for the following alternative definition of Cheeger's constant
\[
h_{\rm DG}(\Omega)=\inf\left\{\frac{P(E)}{|E|}\, :\, E\subseteq \Omega \mbox{ with } |E|>0\right\},
\]
where $P(E)$ is the {\it distributional perimeter} of $E$, in the sense of De Giorgi. This is another possible definition of Cheeger's constant, considered in many papers (in addition to the aforementioned references \cite{Leo} and \cite{Pa2}, we refer for example to \cite{CS, Gri, KLV, KP, LeNeSa}  and \cite{LePr} among others). In general, we have $h_{\rm DG}(\Omega)< h(\Omega)$, see for example \cite[Section 3]{LuMaRu}.
\par
Since the notion of distributional perimeter is not affected by the removal of sets with zero $N-$dimensional Lebesgue measure, we can easily build a counter-example to the validity of \eqref{cheeger_lower} for $h_{\rm DG}$. For example, by taking the following {\it infinite complement comb}
\[
\Omega=\mathbb{R}^2\setminus\{x=(x_1,x_2)\in\mathbb{R}^2\, :\, |x_1|\ge 1, x_2\in\mathbb{Z}\},
\] 
we see that this is a simply connected open set, such that
\[
r_{\Omega}=\sqrt{2}\qquad \mbox{ and }\qquad h_{\rm DG}(\Omega)\le\lim_{n\to\infty} \frac{P((-n,n)\times(-n,n))}{|(-n,n)\times(-n,n)|}=\lim_{n\to\infty}\frac{8\,n}{4\,n^2}=0.
\]
\end{remark}
\medskip
As explained in the Introduction, the constant $h$ plays an important role in Spectral Geometry, because of the celebrated universal lower bound \eqref{cheeger}. 
For general open sets, it is not possible to revert the estimate \eqref{cheeger}, i.e. to get a so-called {\it Buser inequality} of the form
\[
\lambda(\Omega)\le C\,\Big(h(\Omega)\Big)^2.
\]
As already seen, this becomes feasible when suitable geometric assumptions are taken on the open sets. 
Thanks to Corollary \ref{coro:cheeger}, we can easily obtain Buser's inequality for multiply connected open sets in the plane. As simple as the proof is, the result deserves to be stated explicitly. 
\begin{theorem}[Buser's inequality for planar sets]
\label{thm:buser}
For every $\Omega \subseteq \mathbb{R}^2$ open multiply connected set of order $k \in \mathbb{N} \backslash \{0\}$, we have 
\[
\lambda(\Omega)\le \left(\frac{j_{0,1}}{\Theta_{1,1}}\right)^2\,k\,\Big(h(\Omega)\Big)^2,
\]
where $\Theta_{1,1}$ is the same constant as in Theorem \ref{thm:makai-hayman-multiply} and $j_{0,1}$ is the first zero of the Bessel function of the first kind $J_0$ (see for example \cite[page 11]{Hen} for an approximate value).
\end{theorem}
\begin{proof}
We first observe that if $B_{r}(x_0)\subseteq\Omega$, then by monotonicity with respect to set inclusion we have 
\[
\lambda(\Omega)\le \frac{\lambda(B_1)}{r^2}=\frac{(j_{0,1})^2}{r^2}\qquad \mbox{ and }\qquad h(\Omega)\le \frac{\mathcal{H}^{N-1}(\partial B_r(x_0))}{|B_r(x_0)|}=\frac{N}{r}.
\]
For the value of $\lambda(B_1)$ we refer to \cite[Proposition 1.2.14]{Hen}. 
\par
Thus, if $\Omega$ has infinite inradius, from the previous upper bounds we get $\lambda(\Omega)=h(\Omega)=0$ and the result trivially follows. In the case $r_\Omega<+\infty$, it is sufficient to combine \eqref{cheeger_lower} with 
\[
\lambda(\Omega)\le \frac{(j_{0,1})^2}{r^2_\Omega}.
\]
This concludes the proof.
\end{proof}
\begin{remark}
We observe that with exactly the same proof, one can obtain the following Buser--type inequality, for the whole family of generalized principal frequencies: for every $\Omega \subseteq \mathbb{R}^2$ open multiply connected set of order $k \in \mathbb{N} \backslash \{0\}$ and every $1\le q$ which satisfies \eqref{esponenti}, we have
\[
\lambda_{p,q}(\Omega) \leq C\, k^{\frac{p-2}{2} +\frac{p}{q}}\,\Big( h(\Omega)\Big)^{p-2+2\frac{p}{q}},
\]
with the constant $C$ given by 
\[
\frac{\lambda_{p,q}(B_1)}{(\Theta_{1,1})^{p-2+2\frac{p}{q}}}.
\]
Observe that this is now valid for the sub-homogeneous regime $1\le q<p$, as well. In particular, by recalling that $\lambda_{2,1}(\Omega)$ coincides with the reciprocal of the so-called {\it torsional rigidity} $T(\Omega)$, we get the following inequality
\[
\frac{1}{C\,k^2}\,\le \Big(h(\Omega)\Big)^4\,T(\Omega),\qquad \mbox{ with } C=(\Theta_{1,1})^4\,\frac{\pi}{8}.
\]
We also used that $T(B_1)=1/\lambda_{2,1}(B_1)=\pi/8$, in dimension $N=2$.
We refer to \cite{LuMaRu} for a study of this kind of inequality, sometimes called {\it Cheeger--Kohler-Jobin inequality}.
\end{remark}
For every $k\in\mathbb{N}\setminus\{0\}$, we define the sharp constant in the previous Buser inequality, i.e. we set
\[
\mathcal{C}_{\rm B}(k):=\sup\left\{\frac{\lambda(\Omega)}{\Big(h(\Omega)\Big)^2}\, :\, \Omega\subseteq\mathbb{R}^2 \mbox{ multiply connected of order } k \mbox{ with } r_\Omega<+\infty\right\}.
\]
Its precise value is known for $k=1$ and $k=2$ only, see the recent paper \cite{CLS}. For large values of $k$,  
in light of Theorem \ref{thm:buser}, we know that such a constant grows at most like $k$. In the next result, we show that this growth is ``essentially'' sharp.
\begin{proposition}
\label{prop:buserconstant}
The quantity $k\mapsto \mathcal{C}_{\rm B}(k)$ is monotone non-decreasing. Moreover, for every $0<\alpha<1$, we have 
\[
\lim_{k\to\infty} \frac{\mathcal{C}_{\rm B}(k)}{k^\alpha}=+\infty.
\]
\end{proposition}
\begin{proof}
For the monotonicity part, it is sufficient to proceed as follows: if $\Omega\subseteq\mathbb{R}^2$ is admissible for $\mathcal{C}_{\rm B}(k)$, then the set $\widetilde{\Omega}=\Omega\setminus\{x_0\}$ with $x_0\in\Omega$ is admissible for $\mathcal{C}_{\rm B}(k+1)$ and we have
\[
\frac{\lambda(\Omega)}{\Big(h(\Omega)\Big)^2}=\frac{\lambda(\widetilde\Omega)}{\Big(h(\widetilde\Omega)\Big)^2}.
\] 
This is due to the fact that points in dimension $N=2$ have zero $p-$capacity, for every $1\le p\le 2$.
\vskip.2cm\noindent
For the second part of the statement, we are going to exhibit a sequence of open sets $\{\Omega_k\}_{k\ge 2}$ such that each $\Omega_k$ is multiply connected of order $k+1$, it has finite inradius and 
\begin{equation}
\label{francesco}
\lim_{k\to\infty} \frac{\lambda(\Omega_k)}{k^\alpha\,\Big(h(\Omega_k)\Big)^2}=+\infty, \qquad \mbox{ for every } 0<\alpha<1.
\end{equation}
At this aim, we will slightly modify the construction of \cite[Theorem 1.2, point (2)]{BB}. We will produce a sequence of enlarging periodically perforated sets, such that the radius of the perforation shrinks ``not too fast'' as the sets grow.
\par
Let $k\ge 2$ be a natural number and let $\varepsilon_k=k^{-\beta}$ for some fixed $\beta>1/2$, we indicate by
\[
\mathring{Q}_k:=\Big([0,1]\times[0,1]\Big)\setminus B_{\varepsilon_k}\left(\frac{1}{2},\frac{1}{2}\right).
\]
The parameter $\varepsilon_k$ will be the shrinking radius of the perforation.
If we set 
\[
\mathcal{I}_k=\Big\{\mathbf{i}=(i_1,i_2)\in\mathbb{N}^2\, :\, \max\{i_1,\,i_2\}\le {\lfloor \sqrt{k}\rfloor}-1\Big\},
\]
we define
\[
\mathcal{Q}_k=\bigcup_{\mathbf{i}\in \mathcal{I}_k} (\mathring{Q}_k+\mathbf{i}).
\] 
Observe that this is a square with side length $\lfloor \sqrt{k}\rfloor$, containing $(\lfloor \sqrt{k}\rfloor)^2$ equally spaced circular holes of radius $\varepsilon_k$. To this set, whenever $\sqrt{k}\not\in\mathbb{N}$, we attach the perforated horizontal strip
\[
\mathcal{S}_k=\bigcup_{j=0}^{k - \lfloor \sqrt{k} \rfloor^2-1}(\mathring{Q}_k-\mathbf{e}_2+j\,\mathbf{e}_1),
\] 
where $\mathbf{e}_1=(1,0)$ and $\mathbf{e}_2=(0,1)$.
At last, we define 
\[
\Omega_k := \text{int}(\mathcal{Q}_k \cup \mathcal{S}_k),
\]
i.e. the interior of this union (see Figure \ref{fig:optimality_buser_fig}). By construction, this is an open multiply connected set of order $k+1$. Also observe that the inradius $r_{\Omega_k}$ is uniformly bounded, with respect to $k$.
\vskip.2cm\noindent
{\it Estimate for $\lambda(\Omega_k)$.} For every $u \in C^{\infty}_0(\Omega_k)$, by applying Theorem \ref{thm:mazya_poincare} with $d=1/2$ and $D=1$, we get
\[
\begin{split}
\int_{\Omega_k} |\nabla u|^2\,dx&=\int_{\mathcal{Q}_k} |\nabla u|^2\,dx+\int_{\mathcal{S}_k} |\nabla u|^2\,dx\\
&=\sum_{\mathbf{i}\in \mathcal{I}_k}\,\int_{\mathring{Q}_k+\mathbf{i}} |\nabla u|^2\,dx+\sum_{j=0}^{k - \lfloor \sqrt{k} \rfloor^2-1}\,\int_{\mathring{Q}_k-\mathbf{e}_2+j\,\mathbf{e}_1} |\nabla u|^2\,dx\\
&\ge C\,\mathrm{cap}_2(B_{\varepsilon_k};B_1)\,\left(\sum_{\mathbf{i}\in \mathcal{I}_k}\,\int_{\mathring{Q}_k+\mathbf{i}} |u|^2\,dx+\sum_{j=0}^{k - \lfloor \sqrt{k} \rfloor^2-1}\,\int_{\mathring{Q}_k-\mathbf{e}_2+j\,\mathbf{e}_1} |u|^2\,dx\right)\\
&= C\,\mathrm{cap}_2(B_{\varepsilon_k};B_1)\,\int_{\Omega_k} |u|^2\,dx.
\end{split}
\]
By arbitrariness of $u$ and by using \cite[formula (2.2.14)]{Maz} for the relative capacity of a disk, we can infer existence of a constant $C_0>0$ such that 
\begin{equation}
\label{lambda_2_omega_k}
\lambda(\Omega_k)\ge \frac{C_0}{|\log \varepsilon_k|}=\frac{C_0}{\beta\,|\log k|},
\end{equation}
since $\varepsilon_k=k^{-\beta}$.
   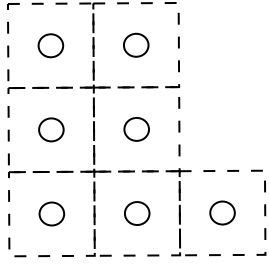
\begin{figure}
   	\centering
   	
   	\tikzset{every picture/.style={line width=0.75pt}} 
   	
   	\begin{tikzpicture}[x=0.3pt,y=0.3pt,yscale=-1,xscale=1]
   		
   		\draw  [dash pattern={on 4.5pt off 4.5pt}] (225.17,76.72) -- (332.77,76.1) -- (333.38,182.25) -- (225.78,182.87) -- cycle ;
   		\draw  [dash pattern={on 4.5pt off 4.5pt}] (117.57,77.34) -- (225.17,76.72) -- (225.78,182.86) -- (118.18,183.48) -- cycle ;
   		\draw  [dash pattern={on 4.5pt off 4.5pt}] (118.19,183.47) -- (225.79,182.85) -- (226.4,288.99) -- (118.8,289.61) -- cycle ;
   		\draw  [dash pattern={on 4.5pt off 4.5pt}] (225.78,182.86) -- (333.39,182.24) -- (334,288.38) -- (226.4,289) -- cycle ;
   		\draw    (152.34,111.55) ;
   		\draw   (156.18,130.1) .. controls (156.18,138.1) and (163.12,144.6) .. (171.68,144.6) .. controls (180.24,144.6) and (187.18,138.1) .. (187.18,130.1) .. controls (187.18,122.09) and (180.24,115.6) .. (171.68,115.6) .. controls (163.12,115.6) and (156.18,122.09) .. (156.18,130.1) -- cycle ;
   		\draw   (263.77,129.48) .. controls (263.77,137.49) and (270.71,143.98) .. (279.27,143.98) .. controls (287.83,143.98) and (294.77,137.49) .. (294.77,129.48) .. controls (294.77,121.48) and (287.83,114.98) .. (279.27,114.98) .. controls (270.71,114.98) and (263.77,121.48) .. (263.77,129.48) -- cycle ;
   		\draw   (156.8,236.23) .. controls (156.8,244.24) and (163.74,250.73) .. (172.3,250.73) .. controls (180.86,250.73) and (187.8,244.24) .. (187.8,236.23) .. controls (187.8,228.22) and (180.86,221.73) .. (172.3,221.73) .. controls (163.74,221.73) and (156.8,228.22) .. (156.8,236.23) -- cycle ;
   		\draw   (264.39,235.62) .. controls (264.39,243.63) and (271.33,250.12) .. (279.89,250.12) .. controls (288.45,250.12) and (295.39,243.63) .. (295.39,235.62) .. controls (295.39,227.61) and (288.45,221.12) .. (279.89,221.12) .. controls (271.33,221.12) and (264.39,227.61) .. (264.39,235.62) -- cycle ;
   		\draw  [dash pattern={on 4.5pt off 4.5pt}] (118.8,289.61) -- (226.4,288.99) -- (227.02,395.13) -- (119.41,395.75) -- cycle ;
   		\draw  [dash pattern={on 4.5pt off 4.5pt}] (226.4,288.99) -- (334.01,288.37) -- (334.62,394.51) -- (227.02,395.13) -- cycle ;
   		\draw  [dash pattern={on 4.5pt off 4.5pt}] (334.01,288.37) -- (441.61,287.75) -- (442.22,393.89) -- (334.62,394.51) -- cycle ;
   		\draw   (265.01,341.75) .. controls (265.01,349.76) and (271.95,356.25) .. (280.51,356.25) .. controls (289.07,356.25) and (296.01,349.76) .. (296.01,341.75) .. controls (296.01,333.74) and (289.07,327.25) .. (280.51,327.25) .. controls (271.95,327.25) and (265.01,333.74) .. (265.01,341.75) -- cycle ;
   		\draw   (157.41,342.37) .. controls (157.41,350.38) and (164.35,356.87) .. (172.91,356.87) .. controls (181.47,356.87) and (188.41,350.38) .. (188.41,342.37) .. controls (188.41,334.36) and (181.47,327.87) .. (172.91,327.87) .. controls (164.35,327.87) and (157.41,334.36) .. (157.41,342.37) -- cycle ;
   		\draw   (372.62,341.13) .. controls (372.62,349.14) and (379.56,355.63) .. (388.12,355.63) .. controls (396.68,355.63) and (403.62,349.14) .. (403.62,341.13) .. controls (403.62,333.12) and (396.68,326.63) .. (388.12,326.63) .. controls (379.56,326.63) and (372.62,333.12) .. (372.62,341.13) -- cycle ;
   	\end{tikzpicture}
   	\caption{The set $\Omega_k$ for $k=7$}
   	\label{fig:optimality_buser_fig}
   \end{figure}  
We now also prove a similar upper bound for $\lambda(\Omega_k)$. We proceed similarly as in the proof of Lemma \ref{lemma:forelli} above. We first observe that
\[
\lambda(\Omega_k)\le \lambda(\mathrm{int}(\mathcal{Q}_k)).
\]
We take the following Lipschitz function defined on $\mathring{Q}_k$ by
\[
u_k(x)=\left(\log\left(\frac{1}{2\,\varepsilon_k}\right)\right)^{-1}\,\min\left\{\log\left(\frac{1}{2\,\varepsilon_k}\right),\log\left(\dfrac{1}{\varepsilon_k}\,\sqrt{\left(x_1-\dfrac{1}{2}\right)^2+\left(x_2-\dfrac{1}{2}\right)^2}\right)\right\}.
\]
Observe that this identically vanishes on $\partial B_{\varepsilon_k}(1/2,1/2)$ and coincides with $1$ on $((0,1)\times(0,1))\setminus B_{1/2}(1/2,1/2)$, see Figure \ref{fig:trombetta}.
\begin{figure}
\includegraphics[scale=.3]{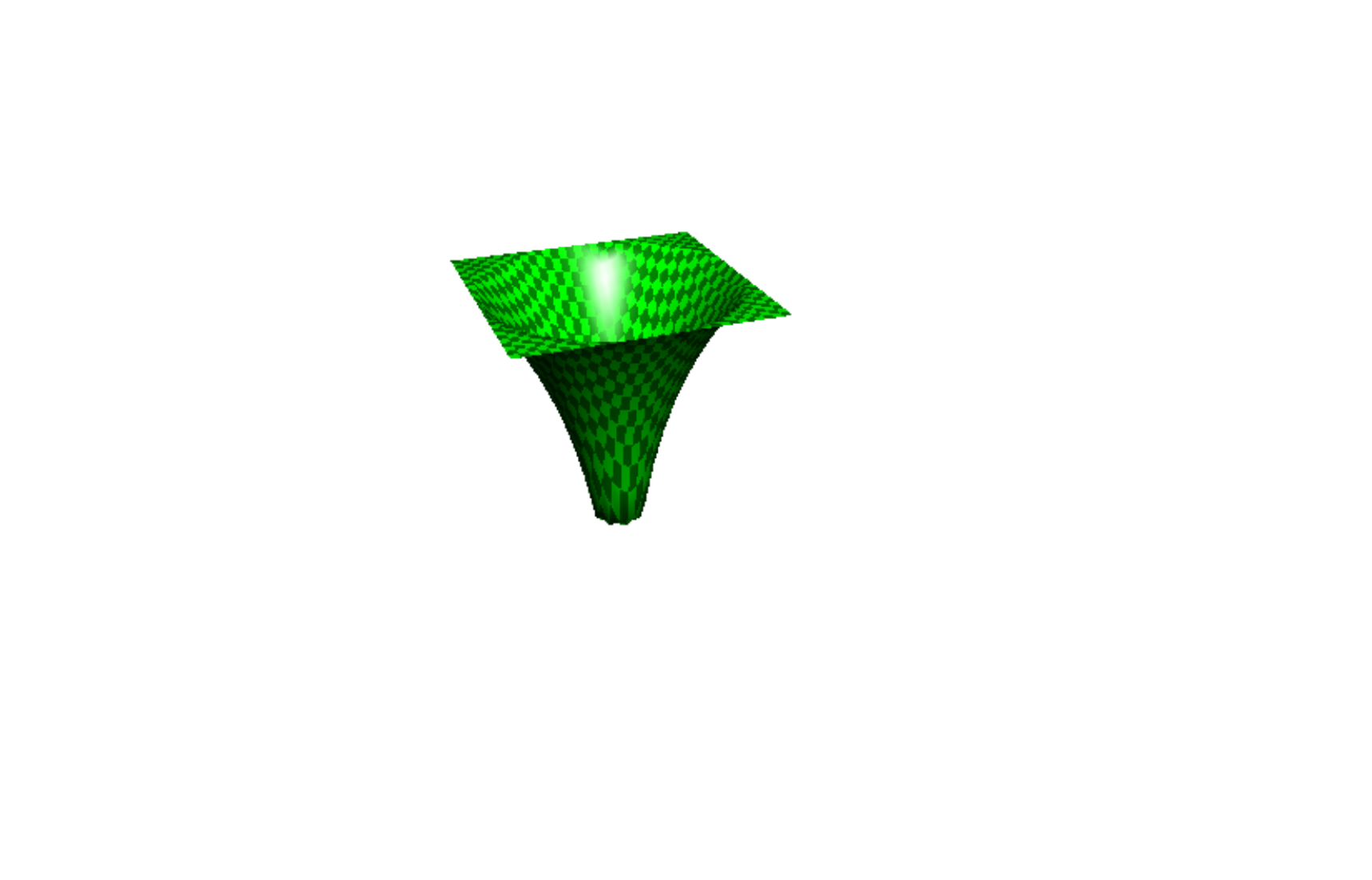}
\caption{The graph of the funnel--type function $u_k$.}
\label{fig:trombetta}
\end{figure}
Then, we periodically repeat it, i.e. we consider
\[
U_k(x)=\sum_{\mathbf{i}\in \mathcal{I}_k} u_k(x+\mathbf{i}).
\]
Finally, we take $\eta_k$ a $1-$Lipschitz cut-off function such that
\[
0\le \eta_k\le 1,\qquad \eta_k\equiv 1 \mbox{ on }\widetilde{\mathcal{Q}_{k}},\qquad \eta_k\equiv 0 \mbox{ on } \partial\mathcal{Q}_k,
\]
where\footnote{In what follows, we suppose that $k\ge 9$. In view of our scopes, this is not restrictive.}
\[
\widetilde{\mathcal{Q}_k}=\bigcup_{\mathbf{i}\in \widetilde{\mathcal{I}}_k} (\mathring{Q}_k+\mathbf{i}),\qquad \mbox{ with }\widetilde{\mathcal{I}}_k=\Big\{\mathbf{i}=(i_1,i_2)\in\mathbb{N}^2\, :\, 1\le \max\{i_1,\,i_2\}\le {\lfloor \sqrt{k}\rfloor}-2\Big\},
\] 
see Figure \ref{fig:optimality_buser_fig2}.
It is easy to see that $\varphi=\eta_k\,U_k\in W^{1,2}_0(\mathrm{int}(\mathcal{Q}_k))$. Thus, by definition of $\lambda$, we have
\begin{equation}
\label{sigillo}
\sqrt{\lambda(\Omega_k)}\le \sqrt{\lambda(\mathrm{int}(\mathcal{Q}_k))}\le \frac{\displaystyle \left(\int_{\mathcal{Q}_k} |\nabla \eta_k|^2\,|U_k|^2\,dx\right)^\frac{1}{2}+\left(\int_{\mathcal{Q}_k} |\nabla U_k|^2\,|\eta_k|^2\,dx\right)^\frac{1}{2}}{\displaystyle\left(\int_{\mathcal{Q}_k}|\eta_k\,U_k|^2\,dx\right)^\frac{1}{2}}
\end{equation}
By using the properties of both $U_k$ and $\eta_k$, we have
\[
\int_{\mathcal{Q}_k}|\eta_k\,U_k|^2\,dx\ge \int_{\widetilde{\mathcal{Q}}_k}|U_k|^2\,dx=\Big({\lfloor \sqrt{k}\rfloor}-2\Big)^2\,\int_{(0,1)\times(0,1)} |u_k|^2\,dx
\]
and
\[
\int_{\mathcal{Q}_k} |\nabla U_k|^2\,|\eta_k|^2\,dx\le \int_{\mathcal{Q}_k} |\nabla U_k|^2\,dx=\Big(\lfloor \sqrt{k}\rfloor\Big)^2\,\int_{(0,1)\times(0,1)}|\nabla u_k|^2\,dx.
\]
Similarly, by recalling that $\eta_k$ is constant on $\widetilde{\mathcal{Q}_k}$, we have 
\[
\begin{split}
\int_{\mathcal{Q}_k} |\nabla \eta_k|^2\,|U_k|^2\,dx&=\int_{\mathcal{Q}_k\setminus\widetilde{\mathcal{Q}_k}} |\nabla \eta_k|^2\,|U_k|^2\,dx\le \int_{\mathcal{Q}_k\setminus\widetilde{\mathcal{Q}_k}}|U_k|^2\,dx\\
&=\left[\Big(\lfloor \sqrt{k}\rfloor\Big)^2-\Big({\lfloor \sqrt{k}\rfloor}-2\Big)^2\right]\,\int_{(0,1)\times(0,1)}|u_k|^2\,dx.
\end{split}
\]
We still need to compute the $W^{1,2}$ norm of $u_k$. From its definition, there exists a constant $C_1>0$ such that we have
\[
\begin{split}
\int_{\mathring{Q}_k}|u_k|^2\,dx&\ge \frac{1}{\big(\log(2\,\varepsilon_k)\big)^2}\int_{B_{1/2}\setminus B_{\varepsilon_k}}\left(\log\left(\dfrac{|x|}{\varepsilon_k}\right)\right)^2\,dx\\
&=\frac{\varepsilon_k^2}{\big(\log(2\,\varepsilon_k)\big)^2}\int_{B_{1/2\varepsilon_k}\setminus B_{1}}\log^2|y|\,dy=\frac{2\,\pi\,\varepsilon_k^2}{\big(\log(2\,\varepsilon_k)\big)^2}\int_1^\frac{1}{2\,\varepsilon_k}\varrho\,\log^2\varrho\,d\varrho\ge C_1,
\end{split}
\]
for $k$ large enough. As for its gradient, we have 
\[
\begin{split}
\int_{\mathring{Q}_k}|\nabla u_k|^2\,dx&= \frac{1}{\big(\log(2\,\varepsilon_k)\big)^2}\int_{B_{1/2}\setminus B_{\varepsilon_k}}\frac{1}{|x|^2}\,dx\\
&=\frac{2\,\pi}{\big(\log(2\,\varepsilon_k)\big)^2}\,\int_{\varepsilon_k}^\frac{1}{2}\frac{1}{\varrho}\,d\varrho\\
&=\frac{2\,\pi}{\big(\log(2\,\varepsilon_k)\big)^2}\,|\log(2\,\varepsilon_k)|=\frac{2\,\pi}{|\log(2\,\varepsilon_k)|}\le\frac{2\,\pi}{\beta\,|\log k|}.
\end{split}
\]
By spending all these informations in \eqref{sigillo}, we get
\begin{equation}
\label{sigillo2}
\sqrt{\lambda(\Omega_k)}\le \left[\frac{\Big(\lfloor \sqrt{k}\rfloor\Big)^2}{\Big({\lfloor \sqrt{k}\rfloor}-2\Big)^2}-1\right]^\frac{1}{2}+\frac{\lfloor \sqrt{k}\rfloor}{{\lfloor \sqrt{k}\rfloor}-2}\,\sqrt{\frac{2\,\pi}{\beta\,C_1}\,\frac{1}{|\log k|}}.
\end{equation}
By using that
\[
\lim_{k\to\infty}\frac{\lfloor \sqrt{k}\rfloor}{{\lfloor \sqrt{k}\rfloor}-2}=1 \qquad \mbox{ and }\qquad \left[\frac{\Big(\lfloor \sqrt{k}\rfloor\Big)^2}{\Big({\lfloor \sqrt{k}\rfloor}-2\Big)^2}-1\right]^\frac{1}{2}\le \left(\frac{8}{\lfloor \sqrt{k}\rfloor}\right)^\frac{1}{2}\le \frac{4}{\sqrt[4]{k}}, \mbox{ for } k\ge 9,
\]
from \eqref{sigillo2} we finally get that there exists a constant $C_2>0$ such that
\begin{equation}
\label{upper_bound_lambda_piccolo}
\lambda(\Omega_{k}) \leq \frac{C_2}{|\log k|},
\end{equation}
for $k$ sufficiently large.

\begin{figure}
	\centering
	
	\tikzset {_sqwguvlde/.code = {\pgfsetadditionalshadetransform{ \pgftransformshift{\pgfpoint{0 bp } { 0 bp }  }  \pgftransformrotate{0 }  \pgftransformscale{2 }  }}}
	\pgfdeclarehorizontalshading{_xt7319nud}{150bp}{rgb(0bp)=(1,1,1);
		rgb(37.5bp)=(1,1,1);
		rgb(62.5bp)=(0,0,0);
		rgb(100bp)=(0,0,0)}
	\tikzset{_0ygkmlxpu/.code = {\pgfsetadditionalshadetransform{\pgftransformshift{\pgfpoint{0 bp } { 0 bp }  }  \pgftransformrotate{0 }  \pgftransformscale{2 } }}}
	\pgfdeclarehorizontalshading{_qyhaojj1f} {150bp} {color(0bp)=(transparent!0);
		color(37.5bp)=(transparent!0);
		color(62.5bp)=(transparent!10);
		color(100bp)=(transparent!10) } 
	\pgfdeclarefading{_xxx39glaw}{\tikz \fill[shading=_qyhaojj1f,_0ygkmlxpu] (0,0) rectangle (50bp,50bp); } 
	
	
	\tikzset {_03gv5bxjb/.code = {\pgfsetadditionalshadetransform{ \pgftransformshift{\pgfpoint{0 bp } { 0 bp }  }  \pgftransformrotate{0 }  \pgftransformscale{2 }  }}}
	\pgfdeclarehorizontalshading{_56tru615w}{150bp}{rgb(0bp)=(1,1,1);
		rgb(37.5bp)=(1,1,1);
		rgb(62.5bp)=(0.9,0.9,0.9);
		rgb(100bp)=(0.9,0.9,0.9)}
	
	
	\tikzset {_ewrtfnkce/.code = {\pgfsetadditionalshadetransform{ \pgftransformshift{\pgfpoint{0 bp } { 0 bp }  }  \pgftransformrotate{0 }  \pgftransformscale{2 }  }}}
	\pgfdeclarehorizontalshading{_lit1siyp3}{150bp}{rgb(0bp)=(1,1,1);
		rgb(37.5bp)=(1,1,1);
		rgb(62.5bp)=(0,0,0);
		rgb(100bp)=(0,0,0)}
	\tikzset{_izv6vge4m/.code = {\pgfsetadditionalshadetransform{\pgftransformshift{\pgfpoint{0 bp } { 0 bp }  }  \pgftransformrotate{0 }  \pgftransformscale{2 } }}}
	\pgfdeclarehorizontalshading{_092cu9sxc} {150bp} {color(0bp)=(transparent!0);
		color(37.5bp)=(transparent!0);
		color(62.5bp)=(transparent!10);
		color(100bp)=(transparent!10) } 
	\pgfdeclarefading{_qlwxn2pc4}{\tikz \fill[shading=_092cu9sxc,_izv6vge4m] (0,0) rectangle (50bp,50bp); } 
	
	
	\tikzset {_f12nxy7sd/.code = {\pgfsetadditionalshadetransform{ \pgftransformshift{\pgfpoint{0 bp } { 0 bp }  }  \pgftransformrotate{0 }  \pgftransformscale{2 }  }}}
	\pgfdeclarehorizontalshading{_evc0lnof8}{150bp}{rgb(0bp)=(1,1,1);
		rgb(37.5bp)=(1,1,1);
		rgb(62.5bp)=(0.9,0.9,0.9);
		rgb(100bp)=(0.9,0.9,0.9)}
	
	
	\tikzset {_jn23th476/.code = {\pgfsetadditionalshadetransform{ \pgftransformshift{\pgfpoint{0 bp } { 0 bp }  }  \pgftransformrotate{0 }  \pgftransformscale{2 }  }}}
	\pgfdeclarehorizontalshading{_mr9ufg5nk}{150bp}{rgb(0bp)=(1,1,1);
		rgb(37.5bp)=(1,1,1);
		rgb(62.5bp)=(0,0,0);
		rgb(100bp)=(0,0,0)}
	\tikzset{_dhrmf5se3/.code = {\pgfsetadditionalshadetransform{\pgftransformshift{\pgfpoint{0 bp } { 0 bp }  }  \pgftransformrotate{0 }  \pgftransformscale{2 } }}}
	\pgfdeclarehorizontalshading{_wbbm0uxrr} {150bp} {color(0bp)=(transparent!0);
		color(37.5bp)=(transparent!0);
		color(62.5bp)=(transparent!10);
		color(100bp)=(transparent!10) } 
	\pgfdeclarefading{_gcvclums4}{\tikz \fill[shading=_wbbm0uxrr,_dhrmf5se3] (0,0) rectangle (50bp,50bp); } 
	
	
	\tikzset {_orvltnqlx/.code = {\pgfsetadditionalshadetransform{ \pgftransformshift{\pgfpoint{0 bp } { 0 bp }  }  \pgftransformrotate{0 }  \pgftransformscale{2 }  }}}
	\pgfdeclarehorizontalshading{_ww3gg76dp}{150bp}{rgb(0bp)=(1,1,1);
		rgb(37.5bp)=(1,1,1);
		rgb(62.5bp)=(0.9,0.9,0.9);
		rgb(100bp)=(0.9,0.9,0.9)}
	
	
	\tikzset {_xq64cedcs/.code = {\pgfsetadditionalshadetransform{ \pgftransformshift{\pgfpoint{0 bp } { 0 bp }  }  \pgftransformrotate{0 }  \pgftransformscale{2 }  }}}
	\pgfdeclarehorizontalshading{_icbu9q5ap}{150bp}{rgb(0bp)=(1,1,1);
		rgb(37.5bp)=(1,1,1);
		rgb(62.5bp)=(0,0,0);
		rgb(100bp)=(0,0,0)}
	\tikzset{_b1jewfwnj/.code = {\pgfsetadditionalshadetransform{\pgftransformshift{\pgfpoint{0 bp } { 0 bp }  }  \pgftransformrotate{0 }  \pgftransformscale{2 } }}}
	\pgfdeclarehorizontalshading{_9i87vpln1} {150bp} {color(0bp)=(transparent!0);
		color(37.5bp)=(transparent!0);
		color(62.5bp)=(transparent!10);
		color(100bp)=(transparent!10) } 
	\pgfdeclarefading{_0rzy2vafw}{\tikz \fill[shading=_9i87vpln1,_b1jewfwnj] (0,0) rectangle (50bp,50bp); } 
	
	
	\tikzset {_vymr3us5c/.code = {\pgfsetadditionalshadetransform{ \pgftransformshift{\pgfpoint{0 bp } { 0 bp }  }  \pgftransformrotate{0 }  \pgftransformscale{2 }  }}}
	\pgfdeclarehorizontalshading{_oennzkq8x}{150bp}{rgb(0bp)=(1,1,1);
		rgb(37.5bp)=(1,1,1);
		rgb(62.5bp)=(0.9,0.9,0.9);
		rgb(100bp)=(0.9,0.9,0.9)}
	\tikzset{every picture/.style={line width=0.75pt}} 
	
	\begin{tikzpicture}[x=0.35pt,y=0.35pt,yscale=-1,xscale=1]
		
		\path  [shading=_xt7319nud,_sqwguvlde,path fading= _xxx39glaw ,fading transform={xshift=2}] (251,106) -- (327,106) -- (327,182) -- (251,182) -- cycle ; 
		\draw  [color={rgb, 255:red, 0; green, 0; blue, 0 }  ,draw opacity=0.57 ][dash pattern={on 4.5pt off 4.5pt}] (251,106) -- (327,106) -- (327,182) -- (251,182) -- cycle ; 
		
		\path  [shading=_56tru615w,_03gv5bxjb] (273.9,144) .. controls (273.9,135.66) and (280.66,128.9) .. (289,128.9) .. controls (297.34,128.9) and (304.1,135.66) .. (304.1,144) .. controls (304.1,152.34) and (297.34,159.1) .. (289,159.1) .. controls (280.66,159.1) and (273.9,152.34) .. (273.9,144) -- cycle ; 
		\draw  [color={rgb, 255:red, 0; green, 0; blue, 0 }  ,draw opacity=0.05 ] (273.9,144) .. controls (273.9,135.66) and (280.66,128.9) .. (289,128.9) .. controls (297.34,128.9) and (304.1,135.66) .. (304.1,144) .. controls (304.1,152.34) and (297.34,159.1) .. (289,159.1) .. controls (280.66,159.1) and (273.9,152.34) .. (273.9,144) -- cycle ; 
		
		\draw  [dash pattern={on 4.5pt off 4.5pt}] (101.6,105) -- (177,105) -- (177,180.4) -- (101.6,180.4) -- cycle ;
		\draw   (124.06,142.7) .. controls (124.06,134.28) and (130.88,127.46) .. (139.3,127.46) .. controls (147.72,127.46) and (154.54,134.28) .. (154.54,142.7) .. controls (154.54,151.12) and (147.72,157.94) .. (139.3,157.94) .. controls (130.88,157.94) and (124.06,151.12) .. (124.06,142.7) -- cycle ;
		\path  [shading=_lit1siyp3,_ewrtfnkce,path fading= _qlwxn2pc4 ,fading transform={xshift=2}] (251,180) -- (327,180) -- (327,256) -- (251,256) -- cycle ; 
		\draw  [color={rgb, 255:red, 0; green, 0; blue, 0 }  ,draw opacity=0.57 ][dash pattern={on 4.5pt off 4.5pt}] (251,180) -- (327,180) -- (327,256) -- (251,256) -- cycle ; 
		
		\path  [shading=_evc0lnof8,_f12nxy7sd] (273.9,218) .. controls (273.9,209.66) and (280.66,202.9) .. (289,202.9) .. controls (297.34,202.9) and (304.1,209.66) .. (304.1,218) .. controls (304.1,226.34) and (297.34,233.1) .. (289,233.1) .. controls (280.66,233.1) and (273.9,226.34) .. (273.9,218) -- cycle ; 
		\draw  [color={rgb, 255:red, 0; green, 0; blue, 0 }  ,draw opacity=0.05 ] (273.9,218) .. controls (273.9,209.66) and (280.66,202.9) .. (289,202.9) .. controls (297.34,202.9) and (304.1,209.66) .. (304.1,218) .. controls (304.1,226.34) and (297.34,233.1) .. (289,233.1) .. controls (280.66,233.1) and (273.9,226.34) .. (273.9,218) -- cycle ; 
		
		\path  [shading=_mr9ufg5nk,_jn23th476,path fading= _gcvclums4 ,fading transform={xshift=2}] (176,179) -- (252,179) -- (252,255) -- (176,255) -- cycle ; 
		\draw  [color={rgb, 255:red, 0; green, 0; blue, 0 }  ,draw opacity=0.57 ][dash pattern={on 4.5pt off 4.5pt}] (176,179) -- (252,179) -- (252,255) -- (176,255) -- cycle ; 
		
		\path  [shading=_ww3gg76dp,_orvltnqlx] (198.9,217) .. controls (198.9,208.66) and (205.66,201.9) .. (214,201.9) .. controls (222.34,201.9) and (229.1,208.66) .. (229.1,217) .. controls (229.1,225.34) and (222.34,232.1) .. (214,232.1) .. controls (205.66,232.1) and (198.9,225.34) .. (198.9,217) -- cycle ; 
		\draw  [color={rgb, 255:red, 0; green, 0; blue, 0 }  ,draw opacity=0.05 ] (198.9,217) .. controls (198.9,208.66) and (205.66,201.9) .. (214,201.9) .. controls (222.34,201.9) and (229.1,208.66) .. (229.1,217) .. controls (229.1,225.34) and (222.34,232.1) .. (214,232.1) .. controls (205.66,232.1) and (198.9,225.34) .. (198.9,217) -- cycle ; 
		
		\path  [shading=_icbu9q5ap,_xq64cedcs,path fading= _0rzy2vafw ,fading transform={xshift=2}] (176,105) -- (252,105) -- (252,181) -- (176,181) -- cycle ; 
		\draw  [color={rgb, 255:red, 0; green, 0; blue, 0 }  ,draw opacity=0.57 ][dash pattern={on 4.5pt off 4.5pt}] (176,105) -- (252,105) -- (252,181) -- (176,181) -- cycle ; 
		
		\path  [shading=_oennzkq8x,_vymr3us5c] (198.9,143) .. controls (198.9,134.66) and (205.66,127.9) .. (214,127.9) .. controls (222.34,127.9) and (229.1,134.66) .. (229.1,143) .. controls (229.1,151.34) and (222.34,158.1) .. (214,158.1) .. controls (205.66,158.1) and (198.9,151.34) .. (198.9,143) -- cycle ; 
		\draw  [color={rgb, 255:red, 0; green, 0; blue, 0 }  ,draw opacity=0.05 ] (198.9,143) .. controls (198.9,134.66) and (205.66,127.9) .. (214,127.9) .. controls (222.34,127.9) and (229.1,134.66) .. (229.1,143) .. controls (229.1,151.34) and (222.34,158.1) .. (214,158.1) .. controls (205.66,158.1) and (198.9,151.34) .. (198.9,143) -- cycle ; 
		
		\draw  [dash pattern={on 4.5pt off 4.5pt}] (101.6,181) -- (177,181) -- (177,256.4) -- (101.6,256.4) -- cycle ;
		\draw   (124.06,218.7) .. controls (124.06,210.28) and (130.88,203.46) .. (139.3,203.46) .. controls (147.72,203.46) and (154.54,210.28) .. (154.54,218.7) .. controls (154.54,227.12) and (147.72,233.94) .. (139.3,233.94) .. controls (130.88,233.94) and (124.06,227.12) .. (124.06,218.7) -- cycle ;
		\draw  [dash pattern={on 4.5pt off 4.5pt}] (100.6,29) -- (176,29) -- (176,104.4) -- (100.6,104.4) -- cycle ;
		\draw   (123.06,66.7) .. controls (123.06,58.28) and (129.88,51.46) .. (138.3,51.46) .. controls (146.72,51.46) and (153.54,58.28) .. (153.54,66.7) .. controls (153.54,75.12) and (146.72,81.94) .. (138.3,81.94) .. controls (129.88,81.94) and (123.06,75.12) .. (123.06,66.7) -- cycle ;
		\draw  [dash pattern={on 4.5pt off 4.5pt}] (175.6,29) -- (251,29) -- (251,104.4) -- (175.6,104.4) -- cycle ;
		\draw   (198.06,66.7) .. controls (198.06,58.28) and (204.88,51.46) .. (213.3,51.46) .. controls (221.72,51.46) and (228.54,58.28) .. (228.54,66.7) .. controls (228.54,75.12) and (221.72,81.94) .. (213.3,81.94) .. controls (204.88,81.94) and (198.06,75.12) .. (198.06,66.7) -- cycle ;
		\draw  [dash pattern={on 4.5pt off 4.5pt}] (250.6,29) -- (326,29) -- (326,104.4) -- (250.6,104.4) -- cycle ;
		\draw   (273.06,66.7) .. controls (273.06,58.28) and (279.88,51.46) .. (288.3,51.46) .. controls (296.72,51.46) and (303.54,58.28) .. (303.54,66.7) .. controls (303.54,75.12) and (296.72,81.94) .. (288.3,81.94) .. controls (279.88,81.94) and (273.06,75.12) .. (273.06,66.7) -- cycle ;
		\draw  [dash pattern={on 4.5pt off 4.5pt}] (176.6,255) -- (252,255) -- (252,330.4) -- (176.6,330.4) -- cycle ;
		\draw   (199.06,292.7) .. controls (199.06,284.28) and (205.88,277.46) .. (214.3,277.46) .. controls (222.72,277.46) and (229.54,284.28) .. (229.54,292.7) .. controls (229.54,301.12) and (222.72,307.94) .. (214.3,307.94) .. controls (205.88,307.94) and (199.06,301.12) .. (199.06,292.7) -- cycle ;
		\draw  [dash pattern={on 4.5pt off 4.5pt}] (100.6,257) -- (176,257) -- (176,332.4) -- (100.6,332.4) -- cycle ;
		\draw   (123.06,294.7) .. controls (123.06,286.28) and (129.88,279.46) .. (138.3,279.46) .. controls (146.72,279.46) and (153.54,286.28) .. (153.54,294.7) .. controls (153.54,303.12) and (146.72,309.94) .. (138.3,309.94) .. controls (129.88,309.94) and (123.06,303.12) .. (123.06,294.7) -- cycle ;
		\draw  [dash pattern={on 4.5pt off 4.5pt}] (251.6,255) -- (327,255) -- (327,330.4) -- (251.6,330.4) -- cycle ;
		\draw   (274.06,292.7) .. controls (274.06,284.28) and (280.88,277.46) .. (289.3,277.46) .. controls (297.72,277.46) and (304.54,284.28) .. (304.54,292.7) .. controls (304.54,301.12) and (297.72,307.94) .. (289.3,307.94) .. controls (280.88,307.94) and (274.06,301.12) .. (274.06,292.7) -- cycle ;
		\draw  [dash pattern={on 4.5pt off 4.5pt}] (326.6,29) -- (402,29) -- (402,104.4) -- (326.6,104.4) -- cycle ;
		\draw   (349.06,66.7) .. controls (349.06,58.28) and (355.88,51.46) .. (364.3,51.46) .. controls (372.72,51.46) and (379.54,58.28) .. (379.54,66.7) .. controls (379.54,75.12) and (372.72,81.94) .. (364.3,81.94) .. controls (355.88,81.94) and (349.06,75.12) .. (349.06,66.7) -- cycle ;
		\draw  [dash pattern={on 4.5pt off 4.5pt}] (325.6,105) -- (401,105) -- (401,180.4) -- (325.6,180.4) -- cycle ;
		\draw   (348.06,142.7) .. controls (348.06,134.28) and (354.88,127.46) .. (363.3,127.46) .. controls (371.72,127.46) and (378.54,134.28) .. (378.54,142.7) .. controls (378.54,151.12) and (371.72,157.94) .. (363.3,157.94) .. controls (354.88,157.94) and (348.06,151.12) .. (348.06,142.7) -- cycle ;
		\draw  [dash pattern={on 4.5pt off 4.5pt}] (327.6,180) -- (403,180) -- (403,255.4) -- (327.6,255.4) -- cycle ;
		\draw   (350.06,217.7) .. controls (350.06,209.28) and (356.88,202.46) .. (365.3,202.46) .. controls (373.72,202.46) and (380.54,209.28) .. (380.54,217.7) .. controls (380.54,226.12) and (373.72,232.94) .. (365.3,232.94) .. controls (356.88,232.94) and (350.06,226.12) .. (350.06,217.7) -- cycle ;
		\draw  [dash pattern={on 4.5pt off 4.5pt}] (327.6,256) -- (403,256) -- (403,331.4) -- (327.6,331.4) -- cycle ;
		\draw   (350.06,293.7) .. controls (350.06,285.28) and (356.88,278.46) .. (365.3,278.46) .. controls (373.72,278.46) and (380.54,285.28) .. (380.54,293.7) .. controls (380.54,302.12) and (373.72,308.94) .. (365.3,308.94) .. controls (356.88,308.94) and (350.06,302.12) .. (350.06,293.7) -- cycle ;

	\end{tikzpicture}
	\caption{The set $\widetilde{\mathcal{Q}}_{k}$ for $k=17$: it is made of the ``internal" perforated squares in grey.}
	\label{fig:optimality_buser_fig2}
\end{figure}
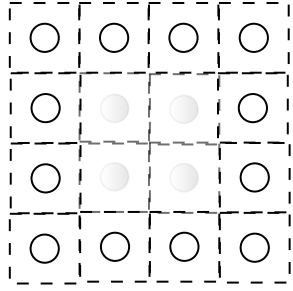

 \vskip.2cm\noindent
{\it Estimate for $h(\Omega_k)$}. By a standard approximation argument (see for example \cite[Proposition 3.3]{Pa}), we can use $\mathcal{Q}_k$ as an admissible set in the definition of $h(\Omega_k)$. This gives
\[
h(\Omega_k)\le \frac{\mathcal{H}^{N-1}(\partial\mathcal{Q}_k)}{|\mathcal{Q}_k|}=\frac{4\,\lfloor \sqrt{k} \rfloor+2\,\pi\,\Big(\lfloor \sqrt{k} \rfloor\Big)^2\,\varepsilon_k}{\Big(\lfloor \sqrt{k} \rfloor\Big)^2(1-\pi\,\varepsilon_k^2)}.
\]
Since by definition we have $\varepsilon_k^2=o(1/k)$, there exists a constant $C_3>0$ such that 
\begin{equation} 
\label{upper_bound_cheeger}
h(\Omega_k)\le \frac{C_3}{\sqrt{k}},
\end{equation}
for $k$ large enough. 
Moreover, for any $k \geq 2$, we have that 
\[
\Omega_k\subseteq \mathbb{R}\times \big(-1,\lfloor \sqrt{k} \rfloor\big).
\]
Thus, by monotonicity with respect to set inclusion and the scaling property of the Cheeger constant, we get
 \begin{equation} \label{cheeger_constant_rectangle}
	h(\Omega_k) \geq \frac{1}{1+\lfloor \sqrt{k} \rfloor}\,h(\mathbb{R}\times (0,1)) = \frac{1}{1+\lfloor \sqrt{k} \rfloor}.
\end{equation}
In the last equality, we used \cite[Theorem 3.1]{KP}.
\vskip.2cm\noindent
{\it Conclusion.} By gathering together the estimates \eqref{lambda_2_omega_k}, \eqref{upper_bound_lambda_piccolo}, \eqref{upper_bound_cheeger} and \eqref{cheeger_constant_rectangle}, we finally obtain
\[
\frac{1}{C}\,\frac{k}{\log k}\le \frac{\lambda(\Omega_k)}{\Big(h(\Omega_k)\Big)^2}\le C\,\frac{k}{\log k},\qquad \mbox{ for $k$ large enough}.
\]
This is enough to establish \eqref{francesco} and conclude the proof.
\end{proof}
As the reader may easily realize, the previous perforated set does not permit to show that
\[
\mathcal{C}_{\rm B}(k)\sim k,\qquad \mbox{ for }k\nearrow \infty.
\]
Such an example may suggest that the sharp growth of $\mathcal{C}_{\rm B}(k)$ could be $k/\log k$, as $k$ goes to $\infty$. In other words, the estimate of Theorem \ref{thm:buser} might perhaps be improved by a logarithmic factor. We leave the following open problem, that we think to be quite interesting.
\begin{open}
Prove or disprove that 
\[
\mathcal{C}_{\rm B}(k)\sim \frac{k}{\log k}\qquad \mbox{ for } k\nearrow \infty.
\]
\end{open}

\appendix

\section{Perforated cubes}
	In what follows, for $R>0$ and $x_0\in\mathbb{R}^n$ we still indicate by $Q_R$ and $Q_R(x_0)$ the cubes given by
	\[
	Q_R=(-R,R)^N\qquad \mbox{ and }\qquad Q_R(x_0)=Q_R+x_0,
	\]
	respectively.
In the proof of Lemma \ref{lemma:forelli}, we used the following result, which is interesting in itself. It concerns the minimization problem
\[	
\Lambda_p(Q_R(x_0)\setminus \overline{B_r(x_0)})=\inf_{u\in \mathrm{Lip}(\overline{Q_R(x_0)})} \left\{\int_{Q_R(x_0)} |\nabla u|^p\,dx\, :\, \|u\|_{L^p(Q_R(x_0))}=1,\, u(x_0)=0 \mbox{ on } \overline{B_r(x_0)}\right\},
\]
with $0\le r<R$. For simplicity, we state it with $x_0=0$.
\begin{lemma} 
\label{lemma:cube_symmetry}
		Let $0\le r<R$ and $1\le p<\infty$. We set 
		\[
		\mathrm{Lip}^{\rm sym}_+(\overline{Q_R})=\Big\{u\in \mathrm{Lip}(\overline{Q_R})\, :\, u\ge 0,\ u\circ \mathcal{R}_i=u \mbox{ for } i=1,\dots,N\Big\},
		\]
		where $\mathcal{R}_i:\mathbb{R}^N\to\mathbb{R}^N$ is the reflection with respect to the hyperplane $\{x\in\mathbb{R}^N\, :\, \langle x,\mathbf{e}_i\rangle=0\}$.
		Then, for every every $1\le p<\infty$ we have
		\[
		\begin{split}
			\inf_{u\in \mathrm{Lip}(\overline{Q_R})} &\left\{\int_{Q_R} |\nabla u|^p\,dx\, :\, \|u\|_{L^p(Q_R)}=1,\, u=0 \mbox{ on } \overline{B_r}\right\}\\
			&=\inf_{u\in \mathrm{Lip}^{\rm sym}_+(\overline{Q_R})} \left\{\int_{Q_R} |\nabla u|^p\,dx\, :\, \|u\|_{L^p(Q_R)}=1,\, u=0 \mbox{ on } \overline{B_r}\right\}.
		\end{split}
		\]
	\end{lemma}
	\begin{proof}
		Obviously, we have
		\[
		\begin{split}
			\inf_{u\in \mathrm{Lip}(\overline{Q_R})} &\left\{\int_{Q_R} |\nabla u|^p\,dx\, :\, \|u\|_{L^p(Q_R)}=1,\, u=0 \mbox{ on }\overline{B_r}\right\}\\
			&\le \inf_{u\in \mathrm{Lip}^{\rm sym}_+(\overline{Q_R})} \left\{\int_{Q_R} |\nabla u|^p\,dx\, :\, \|u\|_{L^p(Q_R)}=1,\, u=0 \mbox{ on } \overline{B_r}\right\}.
		\end{split}
		\]
		In order to prove the reverse inequality, we take $u\in \mathrm{Lip}(\overline{Q_R})$ to be admissible for the variational problem on the left-hand side. Then, we define recursively the non-negative Lipschitz functions
		\[
		\sigma_1=\left(\frac{1}{2}\,|u|^p+\frac{1}{2}\,|u\circ \mathcal{R}_1|^p\right)^\frac{1}{p},
		\] 
		and
		\[
		\sigma_{i+1}=\left(\frac{1}{2}\,\sigma_i^p+\frac{1}{2}\,(\sigma_i\circ \mathcal{R}_{i+1})^p\right)^\frac{1}{p},\qquad \mbox{ for } i=1,\dots,N-1.
		\]
		We claim that for every $i=1,\dots,N$:
\begin{itemize}
\item[(i)]	 $\sigma_i\circ \mathcal{R}_{j}=\sigma_i$, for every $1\le j\le i$;
\vskip.2cm
\item[(ii)] $\|\nabla \sigma_i\|_{L^p(Q_R)}\le \|\nabla u\|_{L^p(Q_R)}$;
\vskip.2cm
\item[(iii)] $\|\sigma_i\|_{L^p(Q_R)}=1$ and $\sigma_i=0$ on $\overline{B_r}$.
\end{itemize}
		In particular, by taking $i=N$, we would get that $\sigma_N$ is admissible for the variational problem on $\mathrm{Lip}^{\rm sym}_+(\overline{Q_R})$ and 
		\[
		\int_{Q_R} |\nabla \sigma_N|^p\,dx\le \int_{Q_R} |\nabla u|^p\,dx.
		\]
		This would be enough to conclude the proof.
		\par
We are left with proving that $\sigma_i$ has the claimed properties. We proceed by induction: for $i=1$, properties (i) and (iii) are straightforward. As for property (ii), this follows from	{\it Benguria's hidden convexity} (originally devised in \cite{Be1, BBL} for $p=2$, extended to $1<p<\infty$ in \cite{BK, DS, KLP, TTU}, see also \cite[Theorem 2.9]{BPZ}), which gives
		\[
			\int_{Q_R} |\nabla \sigma_1|^p\,dx\le \frac{1}{2}\,\int_{Q_R} |\nabla u|^p\,dx+\frac{1}{2}\,\int_{Q_R} |\nabla (u\circ \mathcal{R}_1)|^p\,dx=\int_{Q_R} |\nabla u|^p\,dx,
		\]
		where we used that $\mathcal{R}_1(Q_R)=Q_R$ and that $\mathcal{R}_1$ is a linear isometry .
		Moreover, by construction we still have 
		\[
		\sigma_1=0 \mbox{ on }\overline{B_r(x_0)}\qquad \mbox{ and }\qquad \int_{Q_R(x_0)\setminus B_r(x_0)} |\sigma_1|^p\,dx=1.
		\]
We now take $1\le \ell\le N-1$ and suppose that (i), (ii) and (iii) holds for every $\sigma_1,\dots,\sigma_\ell$. We need to prove that these properties hold for $\sigma_{\ell+1}$, as well.	Again, property (iii) is immediate by construction and by the inductive assumption. For point (i), we have 
\[
\sigma_{\ell+1}=\left(\frac{1}{2}\,\sigma_\ell^p+\frac{1}{2}\,(\sigma_\ell\circ \mathcal{R}_{\ell+1})^p\right)^\frac{1}{p},
\]	
thus for $1\le j\le \ell$
\[
\begin{split}
\sigma_{\ell+1}\circ \mathcal{R}_j&=\left(\frac{1}{2}\,(\sigma_\ell\circ\mathcal{R}_j)^p+\frac{1}{2}\,(\sigma_\ell\circ \mathcal{R}_{\ell+1}\circ\mathcal{R}_j)^p\right)^\frac{1}{p}\\
&=\left(\frac{1}{2}\,\sigma_\ell^p+\frac{1}{2}\,(\sigma_\ell\circ \mathcal{R}_{j}\circ\mathcal{R}_{\ell+1})^p\right)^\frac{1}{p}=\left(\frac{1}{2}\,\sigma_\ell^p+\frac{1}{2}\,(\sigma_\ell\circ\mathcal{R}_{\ell+1})^p\right)^\frac{1}{p}=\sigma_{\ell},
\end{split}
\]
where we exploited the validity of (i) for $1\le j\le \ell$. As for the composition with $\mathcal{R}_{\ell+1}$, we also have
\[
\begin{split}
\sigma_{\ell+1}\circ \mathcal{R}_{\ell+1}&=\left(\frac{1}{2}\,(\sigma_\ell\circ\mathcal{R}_{\ell+1})^p+\frac{1}{2}\,(\sigma_\ell\circ \mathcal{R}_{\ell+1}\circ\mathcal{R}_{\ell+1})^p\right)^\frac{1}{p}\\
&=\left(\frac{1}{2}\,(\sigma_\ell\circ\mathcal{R}_{\ell+1})^p+\frac{1}{2}\,(\sigma_\ell)^p\right)^\frac{1}{p}=\sigma_{\ell},
\end{split}
\]
thanks to the fact that $\mathcal{R}_{\ell+1}\circ\mathcal{R}_{\ell+1}$ is the identity map. This establishes the validity of (i) for $\ell+1$, as well. We still need to verify property (ii): by using again Benguria's hidden convexity, we get
\[
\int_{Q_R} |\nabla \sigma_{\ell+1}|^p\,dx\le \frac{1}{2}\,\int_{Q_R} |\nabla \sigma_\ell|^p\,dx+\frac{1}{2}\,\int_{Q_R} |\nabla (\sigma_\ell\circ \mathcal{R}_{\ell+1})|^p\,dx=\int_{Q_R} |\nabla \sigma_\ell|^p\,dx,
\]
thanks to the fact that $\mathcal{R}_{\ell+1}(Q_R)=Q_R$.
By using that (ii) holds for $\sigma_\ell$, we get the desired conclusion.
\end{proof}

\end{document}